\documentclass{article}
\usepackage[utf8]{inputenc}
\usepackage{fullpage}
\usepackage{amsmath,amsfonts,amssymb,amsthm,enumerate,mdframed,hyperref,cleveref,algorithmic,algorithm,multirow,import,tikz,booktabs,xcolor,soul}
\usetikzlibrary{arrows,positioning, shapes} 
\usepackage{caption,subcaption}
\usepackage{enumitem}
\usepackage{authblk}

\newtheorem{theorem}{Theorem}
\newtheorem{proposition}[theorem]{Proposition}
\newtheorem{lemma}[theorem]{Lemma}

\newtheorem{definition}[theorem]{Definition}
\newtheorem{corollary}[theorem]{Corollary}
\newtheorem{example}[theorem]{Example}

\mathcode`*=\numexpr\mathcode`*-"2000\relax

\definecolor{darkgreen}{rgb}{0,0.6,0}

\providecommand{\F}{\mathbb{F}}
\providecommand{\N}{\mathbb{N}}

\DeclareMathOperator{\PG}{PG}

\newcommand{\cP}{\mathcal{P}}
\newcommand{\cM}{\mathcal{M}}

\newcommand{\gaussm}[3]{\genfrac{[}{]}{0pt}{}{#1}{#2}_{#3}}

\title{Bounds on the minimum distance of locally recoverable codes}
\author{Sascha Kurz}
\affil{Mathematisches Institut, Universität Bayreuth, D-95440 Bayreuth, Germany, sascha.kurz@uni-bayreuth.de \\ \quad and \quad \\ 
Department of Data Science, Friedrich-Alexander-Universitä\"at Erlangen-N\"urnberg, D-91058 Erlangen, Germany, sascha.kurz@fau.de}

\date{}

\begin{document}

\maketitle

\begin{abstract}
  \noindent
  We consider locally recoverable codes (LRCs) and aim to determine the smallest 
  possible length $n=n_q(k,d,r)$ of a linear $[n,k,d]_q$-code with locality $r$. 
  For $k\le 7$ we exactly determine all values of $n_2(k,d,2)$ and for $k\le 6$ 
  we exactly determine all values of $n_2(k,d,1)$. For the ternary field we also 
  state a few numerical results. As a general result we prove that $n_q(k,d,r)$ 
  equals the Griesmer bound if the minimum Hamming distance $d$ is sufficiently 
  large and all other parameters are fixed.

  \smallskip  
  
  \noindent
  \textbf{Mathematics Subject Classification:} 94B27, 94B05\\ 
  \textbf{Keywords:} linear codes, locally recoverable codes, data storage, bounds for parameters      
\end{abstract}

\section{Introduction}
\label{sec_introduction}
A code over a finite alphabet is called \emph{locally recoverable} (LRC) if every symbol in the encoding is a function
of a small number of (at most $r$) other symbols. They have e.g.\ applications in distributed storage and communications. 
Here we will consider linear codes over a finite field $\F_q$ as alphabet. An $[n,k,d]_q$-code $C$ is a $k$-dimensional 
subspace of $\F_q^n$ with minimum Hamming distance (at least) $d$. We also speak of $[n,k]_q$-codes if we do not want to specify 
the minimum Hamming distance. A code symbol is said to have $r$-locality if it can be repaired from at most $r$ other code symbols. 
An $(n,k,r)_q$-LRC is an $[n,k]_q$-code with $r$-locality for all code symbols. For given parameters $n$, $k$ and alphabet size 
$q$ one would like to have codes with small locality $r$, allowing e.g.\ a fast recovery process when the code is used for
distributed storage, and a large minimum Hamming distance $d$, in order to deal with possible errors in the transmission. However, 
there is a natural tradeoff between minimizing $r$ and maximizing $d$. To this end we mention the bound
\begin{equation}
  d\le n-k -\left\lceil\frac{k}{r}\right\rceil+2 \label{ie_upper_bound_singleton}
\end{equation}
by Gopalan et~al.\ \cite{gopalan2012locality}, which reduces to the classical Singleton bound when $r=k$. Surely we have $1\le r\le k$ and MDS codes attain the 
bound for $r=k$. Replicating each symbol in an MDS code twice yields codes attaining the bound for $r=1$. If the field size is sufficiently large then the upper 
bound (\ref{ie_upper_bound_singleton}) can always be attained with equality \cite{gopalan2012locality}. Further constructions requiring smaller field sizes can 
e.g.\ be found in \cite{tamo2014family}, i.e.\ $q\ge n+1$ assuming that $r+1$ divides $n$, see also \cite{luo2018optimal,tamo2015cyclic}. However, small finite 
fields as alphabets are often desirable for practical reasons \cite{goparaju2014binary}. In \cite{tamo2016optimal} the search for constructions meeting the
stated bound (\ref{ie_upper_bound_singleton}) with equality was stated as an open problem. Constructions for the binary case can e.g.\ be found in \cite{goparaju2014binary,huang2015cyclic,huang2016binary}. 
For $q\in\{2,3,4\}$ all cases where Inequality~(\ref{ie_upper_bound_singleton}) can be attained with equality were characterized in \cite{hao2016some}, 
\cite{hao2017optimal}, and \cite{xi2022optimal}, respectively. For $q\ge n+1$ codes with $d= n-k -\left\lceil\tfrac{k}{r}\right\rceil+1$, i.e.\ one less than the
upper bound (\ref{ie_upper_bound_singleton}), indeed exist \cite{tamo2014family}. Using the function $k_{\operatorname{opt}}^q(n,d)$ for the maximum possible dimension $k$ of an 
$[n,k,d]_q$-code the bound
\begin{equation}
  k\le \min_{t\in\mathbb{N}} \left\{rt+ k_{\operatorname{opt}}^q(n-t(r+1),d)\right\}
\end{equation}
for $[n,k,d]_q$-codes with locality $r$ was obtained in \cite{cadambe2015bounds} and e.g.\ used in \cite{hao2020bounds} to conclude explicit parametric bounds on codes with a given locality.

\medskip

Another classical bound for linear codes is the Griesmer bound \cite{griesmer1960bound} relating the minimal possible length $n_q(k,d)$ of an $[n,k,d]_q$-code to its other
parameters:
\begin{equation}
  n_q(k,d)\ge g_q(k,d):=\sum_{i=0}^{k-1}\left\lceil\frac{d}{q^i}\right\rceil\!\!.\label{ie_upper_bound_griesmer}
\end{equation}
Solomon and Stifler gave a construction showing that this bound can be attained for any given parameters $k$ and $q$ if $d$ is sufficiently large \cite{solomon1965algebraically}. Consequently, 
the determination of the function $n_q(k,\cdot)$ becomes a finite problem for each pair of parameters $k$, $q$. For $k\le 7$, the function $n_2(k,\cdot)$ has been completely determined 
in \cite{baumert1973note} and \cite{van1981smallest}. After a lot of work of different authors, the determination of $n_2(8,d)$ has been completed in \cite{bouyukhev2000smallest}. 
Here we will show that for any given pair of parameters $k$, $q$ there exist $(n,k,2)_q$-LRCs and $(n,k,1)_q$-LRCs with length $n=n_q(k,d)$ and minimum Hamming distance $d$ assuming 
that $d$ is sufficiently large. So, also the determination of the largest possible Hamming distance of $(n,k,2)_q$-  and $(n,k,1)_q$-LRCs becomes a finite problem for any parameters $k$ and $q$.  

In the literature also special subclasses of LRCs, like e.g.\ maximally recoverable codes \cite{chen2007maximally}, have been considered. The required field size of a
maximally recoverable codes was e.g.\ improved in \cite{grezet2019uniform}. There the authors used matroid theory, see also \cite{grezet2017binary,westerback2016combinatorics}, to show 
non-existence of some codes with prescribed parameters by characterising linearity over small fields via forbidden uniform minors. In \cite{prakash2012optimal} the authors introduced the notion 
of $(r,\delta)$-LRCs. Corresponding LP and other bounds can e.g.\ be found in \cite{gruica2023lrcs}. In \cite{grezet2019alphabet} the authors introduced a slight variant of the definition of 
locality, called dimension-locality, and study corresponding bounds.  

The remaining part of this paper is structured as follows. In Section~\ref{sec_preliminaries} we introduce the necessary preliminaries. A geometric reformulation of locality is then discussed in 
Section~\ref{sec_geometric} and used to bound the minimal possible lengths of $[n,k,d]_q$-codes with locality $r$. Then we summarize several constructions and non-existence results for LRCs in 
Section~\ref{sec_constructions}. For several small parameter sets we were able to completely determine the minimum possible length of an $[n,k,d]_q$-code with locality $r\in\{1,2\}$ as a function
of the minimum Hamming distance $d$. We close the paper in Section~\ref{sec_enumeration} by stating enumeration results for $[n,k,d]_q$-codes with locality $r$ for some small parameters.
The focus on exact values for small parameters is thought as a supplement to the existing literature which mainly considers different general bounds. While we cannot draw asymptotic conclusions
from our obtained data, we remark that the number of used nodes $n$ is typically rather small in many real-world applications, see e.g.\ HDFS-Xorbas used by Facebook \cite{asteris2013xoring} 
and Windows Azure storage \cite{huang2012erasure}. Since variants and extensions of LRCs are comprehensive, we refrain from discussing similar results for some of them. 

\section{Preliminaries}
\label{sec_preliminaries}
For a given linear code $C$ we denote its minimum Hamming distance by $d(C)$ and the dual code by $C^\perp$. Apart from the field size $q$, the length $n$, the dimension $k$, and the minimum Hamming 
distance $d$ of a linear $[n,k,d]_q$-code we consider the \emph{locality} as an additional parameter.

\begin{definition}
  A (linear) code $C\subseteq \mathbb{F}_q^n$ has \emph{locality} $r$ if for every coordinate $i\in\{1,\dots,n\}$ there exists a set $S_i\subseteq\{1,\dots,n\}$ with $i\notin S_i$, $\left|S_i\right|\le r$, 
  and if $c_j=c'_j$ for all $j\in S_i$ for two codewords $c,c'\in C$ then we have $x_i=y_i$.
\end{definition} 
We also speak of a \emph{locally recoverable code} (LRC). The set $S_i$ is called a \emph{recovery set} for coordinate $i$. Denoting the projection map onto the coordinates from some set
$S\subseteq\{1,\dots,n\}$ by $\pi_S$, we can introduce \emph{recovery functions} $f_i\colon \pi_{S_i}(C)\to \mathbb{F}_q$ satisfying $f_i(\pi_{S_i}(c))=c_i$ for all $c\in C$ and all $1\le i\le n$. 

\begin{lemma}
  (\cite[Lemma 10]{agarwal2018combinatorial})
  Let $C$ be an $[n,k]_q$-code and $S_i$ a recovery set for coordinate $1\le i\le n$ with recovery function $f_i$. Then, $f_i$ is an $\mathbb{F}_q$-linear map.
\end{lemma}
So, we can express the locality of a given linear code $C$ via the existence of certain codewords in its dual code $C^\perp$:
\begin{lemma} 
  \label{lemma_dual_distance}
  A linear $[n,k]_q$-code $C$ has locality $r\ge 1$ iff for every coordinate $1\le i\le n$ there exists a dual codeword $c\in C^\perp$ with weight at most $r+1$ that contains 
  $i$ in its support.
\end{lemma}

\begin{definition}
  An $[n, k]_q$-code $C$ is \emph{non-degenerate}, if there does not exist a coordinate $1\le i\le n$ such that $c_i = 0$ for all codewords $c \in C$. We call 
  $C$ projective if the coordinates of the codewords are linearly independent; that is, there exists no coordinate $i\neq j\in\{1,\dots,n\}$ and $\lambda\in\mathbb{F}_q\backslash\{0\}$ 
  such that $c_i = \lambda\cdot c_j$ for every $c \in C$.
\end{definition}
In other words, a linear code $C$ is non-degenerate if every column of a generator matrix $G$ of $C$ is non-zero. If no column of a generator matrix $C$ is a scalar multiple of another 
column, then the corresponding linear code $C$ is projective.

By $\PG(k-1,q)$ we denote the finite projective geometry of dimension $k-1$ and order $q$. A well-known and often exploited interpretation of non-degenerated linear
codes is the following. The columns of a generator matrix $G$ of a non-degenerated $[n,k]_q$-code $C$ may be interpreted as points of the projective space
$\PG(k-1,q)$. In the other direction, a multiset of points $\mathcal{M}$ in $\PG(k-1,q)$ is a mapping from the set of points to the natural integers. Here, for each point $P$ 
we call $\mathcal{M}(P)$ the multiplicity of $P$. The cardinality $|\mathcal{M}|$ of our multiset $\mathcal{M}$ is the sum over all point multiplicities and equals the 
length $n$ of the corresponding (non-degenerate) code $C$. By $\gamma(\mathcal{M})$ we denote the maximum point multiplicity of $\mathcal{M}$. So, we have $d(C^\perp)=2$ iff
$\gamma(\mathcal{M})\ge 2$ and $\gamma(\mathcal{M})=1$ iff $C$ is projective. Using the geometric language we call $2-$, $3-$, $4-$, and $(k-1)$-dimensional subspaces lines, planes, 
solids, and hyperplanes, respectively. So, each dual codeword $c\in C^\perp$ of weight three geometrically corresponds to a triple of points spanning a line and each dual codeword 
of weight $2$ geometrically corresponds to a point with multiplicity at least $2$. We call a multiset of points $\cM$ in $\PG(k-1,q)$ \emph{spanning} if the points with positive 
multiplicity span the ambient space $\PG(k-1,q)$. The multiset of points corresponding to an $[n,k]_q$-code is always spanning. In the other direction we have that a multiset of points in
$\PG(k-1,q)$ might correspond to an $[n,k']_q$-code with $k'<k$. The minimum Hamming distance of $C$ is at least $d$ iff we have $\mathcal{M}(H)\le n-d$ for every 
hyperplane $H$, where $\mathcal{M}(H)$ is defined as the sum of all multiplicities of the points contained in $H$, i.e., $\cM(H)=\sum_{P\le H} \mathcal{M}(P)$. Writing
$\gaussm{n}{k}{q}$ for the number of $k$-dimensional subspaces in $\PG(n-1,q)$ we can state that each $t$-dimensional subspace contains $\gaussm{t}{1}{q}$ points and 
$\PG(n-1,q)$ contains $\gaussm{n}{n-1}{q}=\gaussm{n}{1}{q}$ hyperplanes in total. Since each point is contained in $\gaussm{k-1}{1}{q}$ hyperplanes and each pair of different points 
is contained in $\gaussm{k-2}{1}{q}$ hyperplanes, assuming $k\ge 3$, we have 
\begin{equation}
  (n-d)\cdot\gaussm{k-1}{1}{q}\ge \sum_{H\,:\,P\le H} \cM(H)=n\cdot \gaussm{k-2}{1}{q}+q^{k-2}\cdot \cM(P) \label{eq_point_mult_ub}
\end{equation}       
for every point $P$. Similarly, 
\begin{equation}
  (n-d)\cdot q^{k-1} \ge \sum_{H\,:\,P\not\le H} \cM(H)= q^{k-2}\cdot(n-\cM(P)) \label{eq_point_mult_lb}
\end{equation}
for every point $P$. For two multisets of points $\cM,\cM'\in\PG(k-1,q)$ we write $\cM+\cM'$ for the multiset of points in $\PG(k-1,q)$ with multiplicity $\cM(P)+\cM'(P)$ for every 
point $P$. Similarly, for each multiset $\cM$ in $\PG(k-1,q)$ and each integer $t\ge 1$ we write $t\cdot\cM$ for the multiset of points in $\PG(k-1,q)$ with multiplicity $t\cdot\cM(P)$ 
for every point $P$.   

\section{Geometric reformulation of locality and minimal possible lengths of $\mathbf{[n,k,d]_q}$-codes with locality $\mathbf{r}$}
\label{sec_geometric}

Directly from Lemma~\ref{lemma_dual_distance} we conclude:
\begin{lemma}
  \label{lemma_locality_geometric}
  Let $C$ be a linear $[n,k]_q$-code and $\cM$ be the corresponding multiset of points in $\PG(k-1,q)$ of cardinality $n$. Then, $C$ has locality $r\ge 1$ iff for every point $P$ 
  with positive multiplicity $\cM(P)\ge 1$ either $\cM(P)\ge 2$ or there are $t\le r$ points with positive multiplicity and being different from $P$ that span an $t$-dimensional subspace 
  $S$ with $P\le S$.
\end{lemma}

So, locality is also a geometric property of a multiset of points $\cM$ and we directly say that $\cM$ has locality $r$ is its corresponding code has locality $r$. We remark that 
especially the locality of a linear code does not depend on a specific representation in terms of a generator matrix with is e.g.\ different for private information retrieval (PIR) codes, 
see \cite[Proposition 9]{kurz2021pir}. For small values of $r$ we can also spell out the condition of Lemma~\ref{lemma_locality_geometric} more directly.
\begin{lemma}
  \label{lemma_locality_1}
  Let $C$ be a linear $[n,k]_q$-code and $\cM$ be the corresponding multiset of points in $\PG(k-1,q)$ of cardinality $n$. Then, $C$ has locality $1$ iff every point $P$ 
  with positive multiplicity $\cM(P)\ge 1$ has multiplicity at least $2$. 
\end{lemma}
Or in other words, a non-empty multiset of points has locality $1$ iff no point has multiplicity exactly $1$.    
\begin{lemma}
  \label{lemma_locality_2}
  Let $C$ be a linear $[n,k]_q$-code and $\cM$ be the corresponding multiset of points in $\PG(k-1,q)$ of cardinality $n$. Then, $C$ has locality $2$ iff for every point $P$ 
  with positive multiplicity $\cM(P)\ge 1$ we have $\cM(P)\ge 2$ or there exist points $Q,R$ with $|\{P,Q,R\}|=3$, $\dim(\langle P,Q,R\rangle)=2$, and $\cM(Q),\cM(R)\ge 1$. 
\end{lemma}    
A multiset of points $\cM$ over the binary field $\F_2$ has locality $2$ iff every point with multiplicity $1$ is contained in a full line, i.e., all three points of the line have positive 
multiplicity. 
\begin{lemma}
  \label{lemma_number_dual_codewords_weight_3}
  Let $C$ be an $[n,k]_q$-code with dual minimum distance $d^\perp=3$. If the number of dual codewords of weight $3$ is less than $(q-1)\cdot n/3$, then the locality of $C$ is larger than $2$.
\end{lemma}
\begin{proof}
  Due to Lemma~\ref{lemma_dual_distance} the locality of $C$ is at least $2$. So, due to Lemma~\ref{lemma_locality_2}, any point with positive multiplicity w.r.t.\ the corresponding 
  multiset of points $\cM$ has to be contained in a line that contains at least three points with positive multiplicity. Our $n$ points with positive multiplicity cannot be covered by 
  fewer than $n/3$ such lines. However, since each such line corresponds to $(q-1)$ dual codewords of weight $3$, we obtain a contradiction.  
\end{proof}
    
\medskip

Similar to the notion of $n_q(k,d)$ let $n_q(k,d,r)$ denote the minimal possible length of an $[n,k,d]_q$-code with locality $r$, i.e., the minimum possible cardinality of a spanning multiset of 
points in $\PG(k-1,q)$ with locality $r$. So, clearly we have
\begin{equation}
  n_q(k,d,r)\ge n_q(k,d)\ge g_q(k,d)=\sum_{i=0}^{k-1}\left\lceil\frac{d}{q^i}.\right\rceil
\end{equation}  
and $n_q(k,d,r)\le n_q(k,d,r')$ for all $r\ge r'$. In Theorem~\ref{main_thm} we will show $n_q(k,d,r)=g_q(k,d)$ for sufficiently large $d$ and fixed parameters $q$, $k$, and $r$.      
    
\begin{example}
  \label{example_simplex_code}
  Let $\cM$ be the (multi-)set of points in $\cP(k-1,q)$ where each point has multiplicity exactly $1$. Then, $\cM$ has cardinality $\tfrac{q^k-1}{q-1}=\gaussm{k}{1}{q}$ 
  and each hyperplane $H$ has multiplicity $\cM(H)=\tfrac{q^{k-1}-1}{q-1}=\gaussm{k-1}{1}{q}$, so that $|\cM|-\cM(H)=q^{k-1}$. The corresponding linear code is called $k$-dimensional $q$-ary 
  simplex code and has parameters $[n,k,d]_q=\left[\gaussm{k}{1}{q},k,q^{k-1}\right]_q$. If $k\ge 2$ then $\cM$ has locality $r=2$ since each point has multiplicity $1$ and each point is contained
  in at least one line (whose points also have multiplicity $1$ each). For each integer $t\ge 1$ the multiset of points $t\cdot \cM$ has cardinality $t\cdot\gaussm{k}{1}{q}$ and locality $1$ iff $t\ge 2$. 
  The corresponding linear code, called $t$-fold simplex code, has parameters $[n,k,d]_q=\left[t\cdot \gaussm{k}{1}{q},k,t\cdot q^{k-1}\right]_q$. Note that we have $n=g_q(k,d)$ for all $t\ge 1$. 
\end{example}  

\begin{theorem}
  \label{main_thm}
  Let $q$ be an arbitrary prime power, $k\in\mathbb{N}_{\ge 2}$, and $r\in\mathbb{N}_{\ge 1}$. If $d$ is sufficiently large, then we have $n_q(k,d,r)=g_q(k,d)$.
\end{theorem}      
\begin{proof}
  It suffices to prove the statement for locality $r=1$.  
  Due to \cite{solomon1965algebraically} there exists a constant $d'$ (depending on $k$ and $q$) such that for all $d\ge d'$ we have $n_q(k,d)=g_k(k,d)$.\footnote{For $k\ge 3$, e.g.\ 
  $d\ge (k-2)q^{k-1}-(k-1)q^{k-2}+1$ is sufficient \cite{maruta1997achievement}.} For $d\ge d'$ consider 
  an $[n,k,d]_q$-code with $n=g_q(k,d)$ and let $\cM$ be the corresponding multiset of points in $\PG(k-1,q)$. Setting $t:=\left\lfloor d/q^{k-1}\right\rfloor$, we have $d\ge t\cdot q^{k-1}$ 
  and $n\le (t+1)\cdot \gaussm{k}{1}{q}$. Using this and Inequality~(\ref{eq_point_mult_lb}) we conclude
  $$
    \cM(P) \ge n(1-q)+dq\ge t-q^k  
  $$  
  for every point $P$. So, if $d$ is sufficiently large, then we have $t\ge q^k+2$ and $\cM$ has locality $1$ due to Lemma~\ref{lemma_locality_1}.
\end{proof}
So for every set of parameters $q$, $k$, and $r$ the determination of $n_q(k,d,r)$, as a function of $d$, is a finite problem. In principle we can determine 
the exact value of $n_q(k,d,r)$ for each given set of parameters as the optimum target value of an integer linear program (ILP). We will spell out the details 
for $r\in\{1,2\}$ and leave the general problem as an exercise for the interested reader.
\begin{proposition}
  Let $q$, $k$, and $d$ be arbitrary but fixed parameters. Then, $n_q(k,d,1)$ is given as the optimum target value of the following ILP:
  \begin{eqnarray*}
  \min n && \text{subject to} \\
  \sum_{P\in\mathcal{P}} x_P &=& n \\
  x_{\langle e_i\rangle} &\ge& 1\quad \forall 1\le i\le k\\ 
  \sum_{P\in\mathcal{P}\,:\ P\le H} x_P &\le& n-d \quad\forall H\in\mathcal{H}\\ 
  x_P &\ge& 2y_P\quad \forall P\in\mathcal{P}\\
  x_P &\le& \Lambda y_P\quad \forall P\in\mathcal{P}\\  
  %% x_P &\le& \Lambda\quad\forall P\in \mathcal{P}\\
  x_P &\in& \mathbb{N}\quad\forall P\in \mathcal{P}\\ 
  y_P &\in& \{0,1\}\quad\forall P\in \mathcal{P},
\end{eqnarray*}
where $e_i$ denotes the $i$th unit vector, $\mathcal{P}$ denotes the set of points in $\PG(k-1,q)$, $\mathcal{H}$ denotes the set of hyperplanes in $\PG(k-1,q)$, 
and $\Lambda$ is a sufficiently large constant.
\end{proposition}  
\begin{proof}
  For a feasible solution of the stated ILP we can define a multiset of points $\cM$ via $\cM(P)=x_P\in\mathbb{N}$ for all points $P\in\mathcal{P}$. The cardinality of 
  $\cM$ is given by $\sum_{P\in\mathcal{P}}\cM(P)=\sum_{P\in\mathcal{P}} x_P = n$. Since $\cM(H)=\sum_{P\in\mathcal{P}\,:\ P\le H} x_P \le n-d$ the corresponding linear code $C$ has Hamming distance 
  at least $d$. Since $\cM(\langle e_i\rangle)=x_{\langle e_i\rangle}\ge 1$ the multiset $\cM$ is spanning, i.e., the linear code $C$ has dimension $k$. For any point $P\in\mathcal{P}$ the constraints 
  $x_P \ge 2y_P$ and $x_P \le \Lambda y_P$ are equivalent to $x_P=0$ for $y_P=0$ and to $2\le x_P\le \Lambda$ for $y_P=1$, i.e., we have $\cM(P)\neq 1$. Thus, $\cM$ and $C$ have locality $1$. 
  
  For the other direction consider an $[n,k,q]_q$-code $C'$ and its corresponding multiset of points $\cM'$. Since $\cM'$ is spanning there exists an isomorphic multisets of points 
  $\cM$ with $\cM(\langle e_i\rangle)\ge 1$ for all $1\le i\le k$. Let $C$ denote the $[n,k,d]_q$-code corresponding to $\cM$. Setting $x_P=\cM(P)\in\mathbb{N}$ for all points $P\in \mathcal{P}$ all 
  constraints that do not involve a $y$-variable are satisfied. If $C'$ has locality $1$, so does $C$ and $\cM$. From Lemma~\ref{lemma_locality_1} we conclude $x_P=\cM(P)\neq 1$ for all $P\in\mathcal{P}$. 
  So, if $\cM(P)=0$ we can set $y_P=1$. If $2\le \cM(P)\le \Lambda$ for a sufficiently large constant $\Lambda$ we can set $y_P=1$. With this all constraints are satisfied, i.e., we have constructed
  a feasible solution of the above ILP. We remark that choosing $\Lambda=\left\lceil d/q^{k-1}\right\rceil \cdot \gaussm{k}{1}{q}\ge n$ always works considering a $t$-fold simplex code with 
  $t=\left\lceil d/q^{k-1}\right\rceil$.    
\end{proof}
We remark that we may also start with a rather small value for $\Lambda$. If we find a feasible solution with target value $n$, then we can deduce $n_q(k,d,r)\le n$. Using this $n$ we can 
utilize Inequality~(\ref{eq_point_mult_ub}) to deduce an upper bound for $\Lambda$. 

Using indicator variables $u_P\in\{0,1\}$ for $\cM(P)\ge 1$ and $z_L\in\{0,1\}$ for lines $L$ containing at least three points with positive multiplicity we can adjust the previous ILP to
model multisets of points with locality $2$. 
\begin{proposition}
  Let $q$, $k$, and $d$ be arbitrary but fixed parameters. Then, $n_q(k,d,2)$ is given as the optimum target value of the following ILP:
  \begin{eqnarray*}
  \min n && \text{subject to} \\
  \sum_{P\in\mathcal{P}} x_P &=& n \\
  x_{\langle e_i\rangle} &\ge& 1\quad \forall 1\le i\le k\\ 
  \sum_{P\in\mathcal{P}\,:\ P\le H} x_P &\le& n-d \quad\forall H\in\mathcal{H}\\
  x_P &\ge& u_P\quad \forall P\in\mathcal{P}\\
  x_P &\le& \Lambda u_P\quad \forall P\in\mathcal{P}\\ 
  x_P &\ge& 2y_P\quad \forall P\in\mathcal{P}\\
  \sum_{P\in\mathcal{P}\,:\,P\le L} u_P &\ge& 3z_L\quad\forall L\in\mathcal{L}\\ 
  y_P+\sum_{L\in\mathcal{L}\,:\, P\le L} z_L&\ge& u_P\quad\forall P\in\mathcal{P}\\
  x_P &\in& \mathbb{N}\quad\forall P\in \mathcal{P}\\ 
  y_P &\in& \{0,1\}\quad\forall P\in \mathcal{P}\\
  u_P &\in& \{0,1\}\quad\forall P\in \mathcal{P}\\
  z_L &\in& \{0,1\}\quad\forall L\in \mathcal{L},
\end{eqnarray*}
where $e_i$ denotes the $i$th unit vector, $\mathcal{P}$ denotes the set of points in $\PG(k-1,q)$, $\mathcal{L}$ denotes the set of lines in $\PG(k-1,q)$,  
$\mathcal{H}$ denotes the set of hyperplanes in $\PG(k-1,q)$, and $\Lambda$ is a sufficiently large constant.
\end{proposition}  

\begin{lemma}
  \label{lemma_monotone}
  For each $n\ge n_q(k,d,r)$ there exists an $[n,k,d]_q$-code with locality $r$.
\end{lemma}
\begin{proof}
  Assume $n_q(k,d,r)<\infty$. Let $C'$ be an $[n',k,d]_q$-code with $n'=n_q(k,d,r)$ and locality $r$. Consider the corresponding multiset of points $\cM'$ in $\PG(k-1,q)$ and let $P$ by 
  an arbitrary point with positive multiplicity. Define the multiset of points $\cM$ in $\PG(k-1,q)$ by setting $\cM(Q)=\cM'(Q)$ for all points $Q\neq P$ and $\cM(P)=\cM'(P)+n-n'\ge \cM'(P)$. 
  by construction we have $|\cM|=n$ and $\cM$ has locality $r$. The linear code $C$ corresponding to $\cM$ also has minimum Hamming distance at least $d$.  
\end{proof}

The previous results are in principle sufficient to determine $n_q(k,d,r)$ for small parameters $q$, $k$, and $r$, so that we may just give tables of the obtained computational results. However,
we prefer to give general constructions and non-existence results in the next section first.

\section{Constructions and non-existence results}
\label{sec_constructions}
In this section we want to study some general constructions to upper bound $n_q(k,d,r)$. To this end we denote the $i$th unit vector by $e_i$, whenever the dimension of the ambient space is clear 
from the context, and for each subspace $S$ its characteristic function is denoted by $\chi_S$, i.e., $\chi_S(P)=1$ if $P\le S$ and $\chi_S(P)=0$ otherwise. In a few cases we can also give theoretical 
non-existence proofs for certain parameters to obtain lower bounds for $n_q(k,d,r)$. First, let us define $n'_q(k,r)$ as the minimum length $n$ of an $[n,k]_q$-code with locality $r$. Clearly, we 
have $n_q(k,d,r)\ge n_q(k,r)$. As observed earlier, the maximum possible locality is given by $r=k$ and we have $n_q(k,d,1)\ge n_q(k,d,2)\ge \dots\ge n_q(k,d,k)$.

\begin{proposition}
  \label{prop_r_1_2_without_d}
  For each integer $k\ge 1$ we have $n'_q(k,1)=2k$ and $n'_q(k,2)=\left\lceil\tfrac{3k}{2}\right\rceil$.
\end{proposition}  
\begin{proof}
   Let $\cM$ be a spanning multiset of points in $\PG(k-1,q)$ with locality $1$. Due to Lemma~\ref{lemma_locality_1} every point $P$ with positive multiplicity $\cM(P)\ge 1$ satisfies 
   $\cM(P)\ge 2$. Since $\cM$ is spanning, it contains at least $k$ points with positive multiplicity, so that $|\cM|\ge 2k$. An attaining example is e.g.\ given by 
   $\cM=\sum_{i=1}^k 2\cdot\chi_{\langle e_i\rangle}$.
   
   For the other case let $\cM$ be a spanning multiset of points in $\PG(k-1,q)$ with locality $2$. Due to Lemma~\ref{lemma_locality_2} every point $P$ with positive multiplicity $\cM(P)\ge 1$ 
   either has multiplicity at least $2$ or is contained in a line $L$ that contains at least three points of positive multiplicity. So, let $\left\{L_1,\dots,L_m\right\}$ be the set 
   of those lines. For each index $1\le i \le m$ let $S_i$ be the subspace spanned by the points of the lines $L_1,\dots,L_i$. To also capture the case $m=0$ let us denote the empty space by
   $S_0$. With this we have $\dim(S_0)=0$, $\cM(S_0)=0$, $\dim(S_1)=2$, and $\cM(S_1)\ge 3$. For $i\ge 2$ we show $\cM(S_i)\le \left\lceil 3\dim(S_i)/2\right\rceil$ and 
   $\dim(S_{i-1})\le\dim(S_i)\le \dim(S_{i-1})+2$ by induction. So, if $L_i$ is completely contained in $S_{i-1}$, then we have $\dim(S_i)=\dim(S_{i-1})$ and $\cM(S_i)=\cM(S_{i-1})$. If the intersection of 
   $S_{i-1}$ and $L_i$ is non-empty but $L_i$ is not completely contained in $S_{i-1}$, then $Q:=S_{i-1}\cap L_i$ is a point. So, we have $\dim(S_i)=\dim(S_{i-1})+1$ and $\cM(S_i)\ge \cM(S_{i-1})+2$. 
   If the intersection of $L_i$ and $S_{i-1}$ is empty, then we have $\dim(S_i)=\dim(S_{i-1})+2$ and $\cM(S_i)\ge \cM(S_{i-1})+3$. Now let us consider the set of points $P$ 
   with positive multiplicity that are not contained in $S_m$. By construction, none of these points is contained in a line $L$ that contains at least three points with positive multiplicity. 
   So, we have $\cM(P)\ge 2$ for each such point. Since $\cM$ is spanning we have at least $k-\dim(S_m)$ of those points and conclude
   $$
     |\cM|\ge \left\lceil\tfrac{3\dim(S_m)}{2}\right\rceil+2\left(k-\dim(S_m)\right)\ge \left\lceil\tfrac{3k}{2}\right\rceil.
   $$  
   For even $k$ an attaining example is given by the set of points
   $$
     \left\{\langle e_{2i-1}\rangle,\langle e_{2i}\rangle,\langle e_{2i-1}+e_{2i}\rangle:\, 1\le i\le k/2\right\}
   $$
   and for odd $k$ and example is given the sum of $2\cdot \chi_{\langle e_k\rangle}$ and the characteristic function of the set of points
   $$
     \left\{\langle e_{2i-1}\rangle,\langle e_{2i}\rangle,\langle e_{2i-1}+e_{2i}\rangle:\, 1\le i\le (k-1)/2\right\}.
   $$
\end{proof}
\begin{corollary}
  \label{cor_d_2_r_2}
  For each integer $k\ge 1$ we have $n_q(k,1,2)=n_q(k,2,2)=\left\lceil\tfrac{3k}{2}\right\rceil$.
\end{corollary}
\begin{proof}
  Due to Proposition \ref{prop_r_1_2_without_d} it suffices to state attaining examples and indeed we will just verify that the examples from the proof of Proposition \ref{prop_r_1_2_without_d}
  have minimum Hamming distance $d=2$. 

  If $k$ is even let $\cM=\sum_{i=1}^k \chi_{\langle e_i\rangle} \,+\,\sum_{i=1}^{k/2} \chi_{\langle e_{2i-1}+e_{2i}\rangle}$, so that $|\cM|=\tfrac{3k}{2}$, $\cM$ is spanning, and has locality $r=2$. 
  For each $1\le i\le k/2$ let $L_i$ the the line spanned by the points $\langle e_{2i-1}\rangle$, $\langle e_{2i}\rangle$ , and $\langle e_{2i-1}+e_{2i}\rangle$. For each hyperplane 
  $H$ at most $k/2-1$ lines $L_i$ can be fully contained in $H$ and the others intersect in at most a point, so that $\cM(H)\le \tfrac{3k}{2}-2$. Thus, the linear code $C$ corresponding to $\cM$ 
  has minimum Hamming distance $d=2$.
  
  If $k$ is odd let $\cM=\sum_{i=1}^{k-1} \chi_{\langle e_i\rangle} \,+\,\sum_{i=1}^{(k-1)/2} \chi_{\langle e_{2i-1}+e_{2i}\rangle}\,+\,2\cdot\chi_{\langle e_k\rangle}$, so that 
  $|\cM|=\left\lceil\tfrac{3k}{2}\right\rceil=\tfrac{3k+1}{2}$, $\cM$ is spanning, and has locality $r=2$. For each $1\le i\le (k-1)/2$ let $L_i$ the the line spanned by the points 
  $\langle e_{2i-1}\rangle$, $\langle e_{2i}\rangle$ , and $\langle e_{2i-1}+e_{2i}\rangle$. By $P$ be denote the point $\langle e_k\rangle$. For each hyperplane $H$ that contains $P$ 
  at most $(k-1)/2-1$ lines $L_i$ can be fully contained in $H$ and the others intersect in at most a point, so that $\cM(H)\le \tfrac{3k+1}{2}-2$. For each hyperplane $H$ that does not 
  contain $P$ we also have $\cM(H)\le\tfrac{3k+1}{2}-2$, so that the linear code $C$ corresponding to $\cM$ has minimum Hamming distance $d=2$.
\end{proof}

\begin{proposition}
  For each integer $k\ge 2$ we have $n'_q(k,k)=k+1$.
\end{proposition}
\begin{proof}
  Let $\cM$ be a spanning multiset of points in $\PG(k-1,q)$ with locality $k$. Since $\cM$ is spanning we have $|\cM|\ge k$. Up to symmetry the unique spanning multiset of points in $\PG(k-1,q)$ 
  of cardinality $k$ is given by the set of points $\left\{\langle e_i\rangle\,:\, 1\le i\le k\right\}$ which does have a finite locality. Thus, we have $n'_q(k,k)\ge k+1$. An attaining example 
  is given by the set of points
  $$
    \left\{\left\langle e_i\right\rangle\,:\, 1\le i\le k\right\}\cup\left\{\left\langle\sum_{i=1}^k e_i\right\rangle\right\}.
  $$
  Here we can easily check that each subset of $k$ points spans the entire ambient space.     
\end{proof}
We remark that the point set of our construction is called \emph{projective base} or \emph{frame} in the literature. Due to Proposition~\ref{prop_r_1_2_without_d} we also have 
$n'_q(k,k)=k+1$ for $k=1$ and we may also consider the double-point $2\cdot \chi_{\langle e_1\rangle}$ as a degenerated cases of a projective base.   

For small dimensions $k\le 2$ the determination of $n_q(k,d,r)$ can be resolved completely analytically:
\begin{proposition}
  We have $n_q(1,d,r)=\max\{2,d\}$, $n_q(2,d,1)=2\left\lceil\tfrac{d}{2}\right\rceil+2$ if $d<2q$, $n_q(2,d,1)=d+\left\lceil\frac{d}{q}\right\rceil$ if $d\ge 2q$, and 
  $n_q(2,d,r)=d+\left\lceil\frac{d}{q}\right\rceil$ for $r\ge 2$.
\end{proposition}
\begin{proof}
  In $\PG(1-1,q)$ the unique multiset of points with cardinality $n$ is given by $\cM=n\cdot \chi_{\langle e_1\rangle}$ which does not have a finite locality if $n=1$ and has locality $1$ if $n\ge 2$. 
  Since there are no hyperplanes in $\PG(1-1,q)$ we need to observe that the non-zero weights of the codewords of the corresponding linear code $C$ all are equal to $n$. Thus, we conclude 
  $n_q(1,d,r)=\max\{2,d\}$.
  
  Let $\cM$ be a spanning multiset of points of cardinality $n$ in $\PG(2-1,q)$ and $C$ its corresponding linear code $[n,k]_q$-code. The minimum Hamming distance of $C$ is at least $d$ iff we 
  have $\cM(P)\le n-d$ for every point $P$. Let us uniquely write $d=aq+b$ with $a\in\mathbb{N}$ and $b\in\{0,1,\dots,q-1\}$. With this we have
  \begin{equation}
    n\ge n_q(k,d,r)\ge n_q(k,d)\ge g_q(k,d)=\sum_{i=0}^{k-1} \left\lceil\frac{d}{q^i}\right\rceil=d+\left\lceil\frac{d}{q}\right\rceil=a(q+1)+b+\left\lceil\tfrac{b}{q}\right\rceil.
  \end{equation}   
  Let $L$ denote the ambient space, which is a line in our situation. If $b=0$ and $a\ge 1$, then $a\cdot \chi_L$ attains this bound and has locality $2$ (or $1$ if $a\ge 2$). If $b\ge 1$ and $a\ge 0$, 
  then the multiset of points given by the sum of $a\cdot\chi_L$ and the characteristic function of arbitrary $b+1\le q$ different points on $L$ attains this bound and has locality $2$ if $a\ge 1$ or $b\ge 2$; 
  we even have locality $1$ if $a\ge 2$. So, for locality $r=2$ it remains to consider the case $d=1$ where we can consider the multiset of points $2\cdot\chi_P$ for an arbitrary point $P$. Thus, 
  for $r\ge 2$ we have $n_q(2,d,r)=d+\left\lceil\frac{d}{q}\right\rceil$ and $n_q(2,d,1)=d+\left\lceil\frac{d}{q}\right\rceil$ if $d\ge 2q$.
  
  For locality $r=1$, dimension $k=2$, and $1\le d<2q$ each point $P$ in a multiset $\cM$ with these parameters satisfying $\cM(P)\ge 1$ indeed has to satisfy $\cM(P)\ge 2$. If we have $l\le 2$ points 
  with positive multiplicity, then we have $n\ge 2l +\cM(P)-2\ge 2l$ and $d\le n-\cM(P)\le 2l-2$ for every point $P$ with positive multiplicity. Thus, for even $2\le d< 2q$ we have 
  $n_q(2,d,1)=d+2$ and for odd $1\le d< 2q$ we have $n_q(2,d,1)=d+3$.  
\end{proof}
For $k\ge 3$ we remark that even the determination of $n_q(k,d)$ is a long-standing open problem for $q>9$, so that we do not expect a closed-form solution for $n_q(k,d,r)$ when $k\ge 3$.

A well-known construction for distance-optimal linear codes is due to Solomon and Stiffler \cite{solomon1965algebraically}.
\begin{lemma}
  \label{lemma_construction_solomon_stiffler}
  Let $k\ge 2$, 
  $$
    n=\sigma[k]_q-\sum_{i=0}^{k-2}\varepsilon_i\gaussm{i+1}{1}{q},
  $$
  and 
  $$
    n-d=\sigma[k-1]_q-\sum_{i=1}^{k-2}\varepsilon_i\gaussm{i}{1}{q},
  $$
  where $\sigma\in\N$ and $\varepsilon_i\in\N$ for all $0\le i\le k-2$. 
  If there exist subspaces $S_1,\dots,S_l$ in $\PG(k-1,q)$ such that 
  \begin{equation}
    \label{eq_S_j_dim_sum}
    \#\left\{1\le j\le l\,:\,\dim\!\left(S_j\right)=i\right\}=\varepsilon_{i+1}
  \end{equation}   
  for $1\le i\le k-1$ and 
  \begin{equation}
    \label{ie_S_j_sum_point_multiplicity_restriction}
    \#\left\{1\le j\le l\,:\,P\in S_j\right\}\le \sigma  
  \end{equation}  
  for each point $P$ in $\PG(k-1,q)$, then an $[n,k,d]_q$-code exists.  
\end{lemma}    
In terms of a multiset of points the underlying construction is given by $\cM=\sigma\cdot \chi_V-\sum\limits_{j=1}^l \chi_{S_j}$, where $V$ denotes the ambient space $\PG(k-1,q)$. In the 
literature mostly the case $n=g_k(k,d)$ is considered, while the construction works of course in general. There are also many criteria available in the literature when those subspaces $S_j$ exist 
given the other numerical parameters. Here we will just speak of a \emph{Solomon-Stifler construction of type} $\left[\sigma;\varepsilon_{k-2},\dots,\varepsilon_1,\varepsilon_0\right]$ and will mostly 
leave the existence proof for subspaces satisfying the conditions of Lemma~\ref{lemma_construction_solomon_stiffler} to the reader.

\begin{example}
  Let $\cM$ be the multiset of points in $\PG(k-1,q)$, where $k\ge 2$, obtained from the Solomon-Stifler construction of type $[1;1,0,\dots,0]$ and $C$ its corresponding $[n,k,d]_q$-code. Then, we 
  have $n=\gaussm{k}{1}{q}-\gaussm{k-1}{1}{q}=q^{k-1}$, $d=(q-1)q^{k-2}$, and $n=g_q(k,d)$. Note that the maximum point multiplicity of $\cM$ is $1$ and that very line $L$ that contains two 
  points of positive multiplicity has multiplicity $\cM(L)=q$ since $L$ intersects the subspace $S_i$ in exactly a point. Thus, we have $n_q(k,(q-1)q^{k-2},2)=q^{k-1}$ for all $q\ge 3$.
\end{example}

The just considered set of points is also known under the term \emph{affine subspace}. For the binary case $q=2$, where our construction does not give a multiset of points with locality $2$, we 
can use the well-known result that each code with the above parameters can be obtained from the Solomon-Stifler construction of type $[1;1,0,\dots,0]$, see e.g.\ \cite[Lemma 12]{kurz2021pir}, to 
conclude:
\begin{lemma}
  \label{lemma_reed_muller_parameters}
  For each $k\ge 2$ we have $n_2(k,2^{k-2},2)\ge 2^{k-1}+1$.
\end{lemma}
%% \begin{lemma}
%%   For $k\ge 2$ let $C$ be a $\left[q^{k-1},k,(q-1)q^{k-2}\right]_q$-code and $\cM$ its corresponding multiset of points. Then, $\cM$ is isomorphic to a multiset of points obtained from the Solomon-Stifler 
%%   construction of type $[1;1,0,\dots,0]$. 
%% \end{lemma} 
%% \begin{proof}
%%   From Inequality~(\ref{eq_point_mult_ub}) we directly conclude $\cM(P)\le 1$ for every point $P$.
%% \end{proof}

\begin{lemma}
  \label{lemma_line_construction}
  Let $C'$ be a projective $[n',k',d']_2$-code with $k'\ge 2$. Then, we have $n_2(k'+1,\min\!\left\{2d',n'+1\right\},2)\le 2n'+1$. 
\end{lemma}
\begin{proof}
  Let $\cM'$ be a spanning multiset of points in $\PG(k'-1,2)$ corresponding to $C'$, so that $\cM'$ has maximum point multiplicity $1$. For each hyperplane $H'$ of $\PG(k'-1,2)$ we have 
  $\cM'(H')\le n'-d'$. Let $S$ be a $k'$-dimensional subspace in $\PG(k-1,2)$, where $k=k'+1$, and $\mathcal{P}$ be the set of points in $\PG(k-1,2)$ that arise from an embedding of $\cM'$ in 
  $S$. Choose an arbitrary point $P$ outside of $S$ and let $\mathcal{L}$ be the set of lines spanned by $P$ and each element of $\mathcal{P}$, so that $|\mathcal{L}|=n'$. With this, we define 
  the multiset of points $\cM$ in $\PG(k,2)$ as the characteristic function of the points contained in at least one line in $\mathcal{L}$, so that $|\mathcal{M}|=2n'+1$. For each hyperplane 
  $H$ of $\PG(k-1,2)$ that contains $P$ we have $\cM(H)\le 1+2(n'-d')$ since at most $n'-d'$ lines of $\mathcal{L}$ can be fully contained in $H$. (The other lines intersect the hyperplane
  $H$ just in point $P$.) For every other hyperplane $H$ in $\PG(k-1,2)$, i.e.\ $P$ is not contained in $H$, we have $\cM(H)= n'$ since any line in $\mathcal{L}$ intersects $H$ in precisely 
  a point. Thus, we have $\cM(H)\le \max\!\left\{1+2(n'-d'),n'\right\}$. So, for the linear code $C$ corresponding to $\cM$ is an $[n,k,d]_2$-code with $n=2n'+1$, $k=k'+1$ and
  $d=\min\!\left\{2d',n'+1\right\}$. By construction, every point in $\PG(k-1,2)$ with positive multiplicity w.r.t.\ $\cM$ lies on a line consisting of three points with positive multiplicity each. 
  Thus, $\cM$ as well as $C$ have locality $2$ and we can deduce the stated upper bound. 
\end{proof}

\begin{proposition}
  \label{prop_reed_muller_parameters_exact}
  For each $k\ge 2$ we have $m_2(k,2^{k-2},2)=2^{k-1}+1$.
\end{proposition}
\begin{proof}
  Due to Lemma~\ref{lemma_reed_muller_parameters} it suffices give a construction showing $m_2(k,2^{k-2},2)\le 2^{k-1}+1$. For $k=2$ such an example is given by the characteristic function 
  of a line. %%, so that we can assume $k\ge 3$ in the following. 
  For $k\ge 3$ we apply Lemma~\ref{lemma_line_construction} with the  first order Reed-Muller code with parameters $(n',k',d')=\left(2^{k-2},k-1,2^{k-3}\right)$ (which corresponds to the 
  characteristic function of an affine space in geometrical terms).
  %% 
  %% Let $S$ be an arbitrary $(k-1)$-dimensional subspace in $\PG(k-1,q)$, $E\le S$ be a $(k-2)$-dimensional subspace, and $P$ be a point 
  %% outside of $S$. With this, let $\cM$ be the characteristic function of the set of points consisting of all points that lie on a line trough $P$ and a point of $S\backslash E$. There are 
  %% exactly $2^{k-2}$ such lines, so that $n:=|\cM|=1+2^{k-1}$. For each hyperplane $H$ that contains $P$ exactly $2^{k-3}$ of the lines are fully contained in $H$, so that $\cM(H)=1+2^{k-2}$. 
  %% For each hyperplane $H$ that does not contain $P$ each of the $2^{k-2}$ lines is exactly intersected in a point, so that $\cM(H)=2^{k-2}$. Thus, the linear code $C$ corresponding to $\cM$ 
  %% has minimum distance $d=2^{k-2}$. By construction each point lies on a line whose points all have multiplicity $1$, so that $\cM$ has locality $r=2$.  
\end{proof}

Next we want to give an easy sufficient criterion when a code obtained from the Solomon-Stifler construction has locality $2$:
\begin{lemma}
  \label{lemma_solomon_stifler_r_2}
  Let $\cM$ be a spanning multiset of points in $\PG(k-1,q)$ obtained from the Solomon-Stifler construction with type $\left[\sigma;\varepsilon_{k-2},\dots,\varepsilon_1,\varepsilon_0\right]$. 
  If $\sum_{i=0}^{k-1} \varepsilon_i\cdot\gaussm{i+1}{1}{q}< \sigma\cdot\gaussm{k-1}{1}{q}$, then $\cM$ has locality $r=2$.
\end{lemma} 
\begin{proof}
  Let $\cM=\sigma\cdot \chi_V-\sum\limits_{j=1}^l \chi_{S_j}$ using the notation from Lemma~\ref{lemma_construction_solomon_stiffler}, where $V$ denotes the ambient space $\PG(k-1,q)$. 
  In $\PG(k-1,q)$ each point $P$ is on $\gaussm{k-1}{1}{q}$ many lines and so $\chi_V$ as well as $\sigma\cdot\chi_V$ contain $\gaussm{k-1}{1}{q}$ lines through $P$ in its support. 
  Since $\cM$ arises from $\sigma\cdot\chi_V$ by decreasing point multiplicities by $\sum_{i=0}^{k-1} \varepsilon_i\cdot\gaussm{i+1}{1}{q}<\sigma\cdot \gaussm{k-1}{1}{q}$ and 
  removing a point from the support costs multiplicity $\sigma$, at least one full line through $P$ remains if $\cM(P)\ge 1$.   
\end{proof}

\begin{theorem}
  \label{thm_m_q_2_k_3_r_2}
  For each $t\in \mathbb{N}$ we have $n_2(3,3+4t,2)=6+7t$, $n_2(3,4+4t,2)=7+7t$, $n_2(3,5+4t,2)=10+7t$, and $n_2(3,6+4t,2)=11+7t$. Moreover, we have $n_2(3,1,2)=n_2(3,2,2)=5$. 
\end{theorem}
\begin{proof}
  In $\PG(3-1,2)$ consider Solomon-Stifler constructions with types $[t+1;0,1]$, $[t+1;0,0]$, $[t+2;1,1]$, and $[t+2;1,0]$, respectively. The lengths $n$ and minimum distances $d$ as well
  as the dimension $k=3$ are as stated. Using Lemma~\ref{lemma_solomon_stifler_r_2} we can easily check that all those examples have locality $r=2$ (and in some cases even locality $r=1$). 
  Proposition~\ref{prop_reed_muller_parameters_exact} yields $n_2(3,2,2)=5$, so that it remains to show $n_2(3,1,2)\ge 5$, which is implied by Proposition~\ref{prop_r_1_2_without_d}.     
\end{proof}

\begin{corollary}
  We have $n_2(3,1,2)=g_2(3,1)+2$, $n_2(3,2,2)=g_2(3,2)+1$, and $n_2(3,d,2)=g_2(3,d)$ for all $d\ge 3$.  
\end{corollary}
We remark that we have $n_2(3,d)=g_2(3,d)$ for all $d\ge 1$.

\begin{lemma}
  \label{lemma_construction_d_3}
  For $k\ge 4$ we have $n_2(k,3,2)\le 2k$.
\end{lemma}
\begin{proof}
  For $1\le i\le k-1$ let $L_i=\langle e_i,e_{i+1}\rangle$ and $L_k=\langle e_k,e_1\rangle$. With this, let $\cM$ be the characteristic function of all points that are contained in one of 
  the lines $L_i$, so that $|\cM|=2k$, $\cM$ has locality $r=2$, and $\cM$ is spanning. For each hyperplane $H$ there exists and index $1\le i\le k$ such that $P:=\langle e_i\rangle$ 
  is not contained in $H$. The two lines $L_j$ that contain $P$ intersect $H$ in precisely a point so that $\cM(H)\le 2k-3$. Thus, the linear code $C$ corresponding to $\cM$ has 
  minimum Hamming distance $d$. 
\end{proof}

\begin{theorem}
  \label{thm_m_q_2_k_4_r_2}
  For each $t\in \mathbb{N}$ we have $n_2(4,5+8t,2)=11+15t$, $n_2(4,6+8t,2)=12+15t$, $n_2(4,7+8t,2)=14+15t$, $n_2(4,8+8t,2)=15+7t$, $n_2(4,9+8t,2)=19+15t$, $n_2(4,10+8t,2)=20+15t$, 
  $n_2(4,11+8t,2)=22+15t$, and $n_2(4,12+8t,2)=23+15t$. Moreover, we have $n_2(4,1,2)=n_2(4,2,2)=6$, $n_2(4,3,2)=8$, and $n_2(4,4,2)=9$. 
\end{theorem}
\begin{proof}
  In $\PG(4-1,2)$ consider Solomon-Stifler constructions with types $[t+1;0,1,1]$, $[t+1;0,1,0]$, $[t+1;0,0,1]$, $[t+1;0,0,0]$, $[t+2;1,1,1]$, $[t+2;1,1,0]$, $[t+2;1,0,0]$, and $[t+2;1,0,0]$, 
  respectively. The lengths $n$ and minimum distances $d$ as well as the dimension $k=4$ are as stated. Using Lemma~\ref{lemma_solomon_stifler_r_2} we can easily check that all those examples
  have locality $r=2$ (and in some cases even locality $r=1$).
     
  Proposition~\ref{prop_reed_muller_parameters_exact} yields $n_2(4,4,2)=9$. Proposition~\ref{prop_r_1_2_without_d} yields $n_2(4,3,2)\ge n_2(4,2,2)\ge n_2(4,1,2)\ge 
  \left\lceil\tfrac{3\cdot 4}{2}\right\rceil=6$. Corollary~\ref{cor_d_2_r_2} yields $n_2(4,1,2)=n_2(4,2,2)=6$ and Lemma~\ref{lemma_construction_d_3} yields $n_2(4,3,2)\le 8$.  
  %% Consider the multiset of points $\cM=\chi_{L_1}+\chi_{L_2}$ for two disjoint lines $L_1,L_2$. So, $\cM$ is spanning and has 
  %% cardinality $2\cdot\gaussm{2}{1}{2}=6$. Since each hyperplane $H$ can fully contain at most one of the lines $L_1$, $L_2$ and intersects the other in a point, we have $\cM(H)\le 4$, 
  %% so that the corresonding linear code has minimum Hamming distance $d=6-4=2$. By construction every point with positive multiplicity is contained in a line (here more specifically $L_1$ 
  %% or $L_2$) consisting of three points with positive multiplicity, so that $\cM$ has locality $r=2$. Thus, we have $n_2(4,1,2)=n_2(4,2,2)=6$. 
  %% Next consider the point set $\cM$ 
  %% given by the points of all lines $\left\langle e_i,e_j\right\rangle$ with $1\le i<j\le 4$, $\{i,j\}\neq \{1,2\}$, and $\{i,j\}\neq \{3,4\}$. For the corresponding $[n,k,d]_2$-code 
  %% $C$ we can easily check $n=8$, $k=4$, and $d=3$. By construction every point with positive multiplicity is contained in one of the four lines that are contained in the support of $\cM$, 
  %% so that $n_2(4,3,2)\le 8$.
  
  Due to Lemma~\ref{lemma_monotone} it suffices to assume that $\cM$ is a spanning multiset of points in $\PG(4-1,2)$ with cardinality $7$, locality $2$, $\cM(H)\le 4$ for every hyperplane 
  $H$ and conclude a contradiction. First, assume that $P_1$ is a point 
  with multiplicity at least $2$. If $L$ is a line in the support of $\cM$, then we have $P_1\le L$ since otherwise the hyperplane spanned by $P_1$ and $L$ would have multiplicity at least 
  $5$. However, for any point $P_2$ with positive multiplicity that is not contained in $L$ the hyperplane spanned by $P_2$ and $L$ has multiplicity at least $5$ -- contradiction. So, if
  one point $P_1$ has multiplicity at least two, then all points with positive multiplicity have multiplicity at least two. However, a hyperplane spanned by three such points, that are 
  not contained in a line, has multiplicity at least $6$ -- contradiction. Thus, the maximum point multiplicity of $\cM$ is $1$ and for each point $P$ with positive multiplicity there 
  exists a line $L_P$ in the support of $\cM$. No two different such lines can intersect in a point since otherwise the hyperplane spanned by these two lines would have multiplicity at least 
  $5$. However, since $7$ is not divisible by $3$ the points with positive multiplicity cannot be partitioned into pairwise disjoint lines.
\end{proof}

\begin{corollary}
  We have $n_2(4,1,2)=g_2(4,1)+2$, $n_2(4,2,2)=g_2(4,2)+1$, $n_2(4,2,3)=g_2(4,3)+1$, $n_2(4,2,4)=g_2(4,4)+1$, and $n_2(4,d,2)=g_2(4,d)$ for all $d\ge 5$.  
\end{corollary}
We remark that we have $n_2(4,d)=g_2(4,d)$ for all $d\ge 1$.

\begin{theorem}
  \label{thm_m_q_2_k_5_r_2}
  For each $t\in \mathbb{N}$ we have $n_2(5,9+16t,2)=20+31t$, $n_2(5,10+16t,2)=21+31t$, $n_2(5,11+16t,2)=23+31t$, $n_2(5,12+16t,2)=24+31t$, $n_2(5,13+16t,2)=27+31t$, $n_2(5,14+16t,2)=28+31t$, 
  $n_2(5,15+16t,2)=30+31t$, $n_2(5,16+16t,2)=31+31t$, $n_2(5,17+16t,2)=36+31t$, $n_2(5,18+16t,2)=37+31t$, $n_2(5,19+16t,2)=39+31t$, $n_2(5,20+16t,2)=40+31t$, $n_2(5,21+16t,2)=43+31t$, 
  $n_2(5,22+16t,2)=44+31t$, $n_2(5,23+16t,2)=46+31t$, and $n_2(5,24+16t,2)=47+31t$. Moreover, we have $n_2(5,1,2)=n_2(5,2,2)=8$, $n_2(5,3,2)=10$, $n_2(5,4,2)=11$, $n_2(5,5,2)= 13$, 
  $n_2(5,6,2)=14$, $n_2(5,7,2)= 16$, and $n_2(5,8,2)=17$. 
\end{theorem}
\begin{proof}
  In $\PG(5-1,2)$ consider Solomon-Stifler constructions with types $[t+1;0,1,1,1]$, $[t+1;0,1,1,0]$, $[t+1;0,1,0,1]$, $[t+1;0,1,0,0]$, $[t+1;0,0,1,1]$, $[t+1;0,0,1,0]$, $[t+1;0,0,0,1]$, 
  $[t+1;0,0,0,0]$, $[t+2;1,1,1,1]$,  $[t+2;1,1,1,0]$, $[t+2;1,1,0,1]$, $[t+2;1,1,0,0]$, $[t+2;1,0,1,1]$, $[t+2;1,0,1,0]$, $[t+2;1,0,0,1]$, and $[t+2;1,0,0,0]$, respectively. The lengths 
  $n$ and minimum distances $d$ as well as the dimension $k=5$ are as stated. Using Lemma~\ref{lemma_solomon_stifler_r_2} we can easily check that all those examples have locality $r=2$
  (and in some cases even locality $r=1$).

  Proposition~\ref{prop_reed_muller_parameters_exact} yields $n_2(5,8,2)=17$. Proposition~\ref{prop_r_1_2_without_d} yields $n_2(5,3,2)\ge n_2(5,2,2)\ge n_2(5,1,2)\ge 
  \left\lceil\tfrac{3\cdot 5}{2}\right\rceil=8$. Corollary~\ref{cor_d_2_r_2} yields $n_2(5,1,2)=n_2(5,2,2)=8$. Lemma~\ref{lemma_construction_d_3} implies $n_2(5,3,2)\le 10$. Applying 
  Lemma~\ref{lemma_line_construction} to an $[5,4,2]_2$-code gives $n_2(5,4,2)\le 11$. The generator matrices
  $$
    \begin{pmatrix}
      1111111010000\\
      0001111101000\\
      0110011100100\\
      1011100000010\\
      1000111000001
    \end{pmatrix}\!\!,\quad   
     \begin{pmatrix}
      11111110010000\\
      00011111101000\\
      01100110100100\\
      10101011100010\\
      11010011000001
    \end{pmatrix}  
    \quad\text{and}\quad
    \begin{pmatrix}
    0000000001111111\\
    0001111110000011\\
    0010001110001101\\
    0100010110010110\\
    1000111000101001  
    \end{pmatrix}
  $$
  give $[13,5,5]_2$-, $[14,5,6]_2$-, and $[16,5,7]_2$-codes with locality $2$, so that we have $n_2(5,5,2)\le 13$, $n_2(5,6,2)\le 14$, and $n_2(5,7,2)\le 16$. 
  For the lower bounds we have $n_2(5,3,2)\ge n_2(5,3)=10$, $n_2(5,4,2)\ge n_2(5,4)=11$, $n_2(5,5,2)\ge n_2(5,5)= 13$, and  $n_2(5,6,2)\ge n_2(5,6)=14$. Finally, if $C$ is an $[15,5,7]_2$-code 
  with locality $r=2$, then adding a parity check bit yields an $[16,5,8]_2$-code $C'$. Let $\cM$ and $\cM'$ denote the multisets of points corresponding to $C$ and $C'$, respectively. Given its 
  parameters, the code $C'$ is unique up to isomorphism and can be obtained by the Solomon-Stifler construction of type $[1;1,0,0,0,0]$, see e.g.\ \cite[Lemma 12]{kurz2021pir}. Note 
  that the maximum point multiplicity of $\cM'$ is $1$ and that $\cM'$ does not contain a full line in its support. Since $\cM'$ arises from $\cM$ by increasing the point multiplicity of a (specific) 
  point by one, also the maximum point multiplicity of $\cM$ is $1$ and $\cM'$ does not contain a full line in its support. Thus, $\cM$ cannot have locality $2$ and we have $n_2(5,7,2)=16$.  
\end{proof}

\begin{corollary}
  We have $n_2(5,1,2)=g_2(5,1)+3$, $n_2(5,2,2)=g_2(5,2)+2$, $n_2(5,2,3)=g_2(5,3)+2$, $n_2(5,2,4)=g_2(5,4)+2$, $n_2(5,2,5)=g_2(5,5)+1$, $n_2(5,2,6)=g_2(5,6)+1$, $n_2(5,2,7)=g_2(5,7)+1$, 
  $n_2(5,2,8)=g_2(5,8)+1$, and $n_2(5,d,2)=g_2(5,d)$ for all $d\ge 9$.  
\end{corollary}
We remark that we have $n_2(5,d)=g_2(5,d)$ for all $d\in\mathbb{N}\backslash\{3,4,5,6\}$. Note that the code $C'$ used in the proof of Theorem~\ref{thm_m_q_2_k_5_r_2} to show $n_2(5,7,2)>15$ 
also has to be even, i.e., all of its weights are divisible by $2$. So, for some parameters we might be able to show that the weight distribution of an even $[n,k,d]_2$-code $C'$ (where also $d$ 
is even) is unique and can be determined using theoretical methods. So, for each $[n-1,k,d-1]_2$-code $C$ adding a parity check bit yields such an even $[n,k,d]_2$-code with known weight distribution.
Applying the MacWilliams transform we then compute also the dual weight distribution of $C'$. To this end let us slightly generalize Lemma~\ref{lemma_number_dual_codewords_weight_3} and the proof idea for 
$n_2(5,7,2)>15$.
\begin{lemma}
  Let $C$ be an $[n,k,d]_2$-code with odd minimum distance $d$ and $C'$ be the $[n+1,k,d+1]_2$-code obtained from $C$ by adding a parity check bit. If $C'$ has dual minimum distance $d^\perp=3$ 
  and at less than $n/3$ dual codewords of weight $3$, then the locality of $C$ is larger than $2$.  
\end{lemma}
\begin{proof}
  Let $\cM$ and $\cM'$ be the multisets of points corresponding to $C$ and $C'$, respectively. Note that $\cM'$ arises from $\cM$ by increasing the point multiplicity of a certain point $P$ by $1$. 
  Using the fact that the dual minimum distance of $C'$ is $3$ we conclude that both $\cM'$ and $\cM$ have a maximum point multiplicity of $1$. Since $C'$ contains less than $n/3$ dual codewords 
  of weight $3$, at most $n-1$ points with positive multiplicity in $\cM'$ can be contained in a line that is fully contained in the support of $\cM'$. Thus, there exists a point $Q\neq P$ 
  with $\cM(Q)=1$ that is not contained in line $L$ that is fully contained in the support of $\cM$. Using Lemma~\ref{lemma_locality_2} we conclude that the locality of $\cM$ and $C$ is at least $3$. 
\end{proof} 
 
\begin{lemma}
  \label{lemma_d_4_construction}
  For each $t\in\mathbb{N}_{\ge 2}$ we have $n_q(2t,4,2)\le 3t+3$ and $n_2(2t+1,4,2)\le 3t+5$. 
\end{lemma} 
\begin{proof}
  Consider the $t$ triples of points $L_i=\left\{ e_{2i-1},e_{2i},e_{2i-1}+e_{2i}\right\}$ for $1\le i\le t$, the triple of points $L'=\left\{ \sum_{i=1}^t e_{2i-1},\sum_{i=1}^t e_{2i},\sum_{i=1}^{2t} e_i \right\}$, 
  and the triple of points $L''=\left\{ \sum_{i=1}^{t+1} e_{2i-1},\sum_{i=1}^t e_{2i},\sum_{i=1}^{2t+1} e_i\right \}$. (Over $\F_2$ these triples are full lines.) With this let $\cM=\chi_{L'}+\sum_{i=1}^t \chi_{L_i}$ 
  and $\cM'=2\cdot \chi_{P}+\chi_{L''}+\sum_{i=1}^t \chi_{L_i}$, where $P=\langle e_{2t+1}\rangle$. Note that $\cM$ spans $\PG(2t-1,q)$, $\cM'$ spans $\PG(2t,q)$, $|\cM|=3t+3$, $|\cM'|=3t+5$, and both multisets of points 
  have locality $r=2$. So, it remains to upper bound the multiplicities of the hyperplanes of the respective ambient spaces. By construction, the multiplicities of $H\cap L'$, $H\cap L''$, and 
  $H\cap L_i$, where $1\le i\le t$, are not equal to $2$ for each hyperplane $H$. 
  
  Let us first consider $\cM$ in $\PG(2t-1,q)$. Note that a hyperplane $H$ cannot fully contain all $L_i$ for $1\le i\le t$. 
  Due to symmetry we assume that $L_1$ is not fully contained in $H$. If also another triple $L_i$ with $2\le i\le t$ is not fully contained in $H$, then we have $\cM(H)\le 3t+3-2\cdot 2$. So, let us assume 
  that $H$ fully contains all triples $L_i$ for $2\le i\le t$. If $\left|H\cap L_1\right|=1$, then due to symmetry we assume $\left\langle e_1\right\rangle\le H$, so that $H$ is uniquely determined and 
  we have $\cM(H)\le 3t+3-2\cdot 2$ since $L'$ is not fully contained in $H$. If $\left|H\cap L_1\right|=0$ then $L'$ cannot be fully contained in $H$, so that $\cM(H)\le 3t+3-5$.  
  
  Next consider $\cM'$ in $\PG(2t,q)$. There is a unique hyperplane $H$ that fully contains $L_i$ for $1\le i\le t$. Here we have $\cM'(H)\le 3t+5 -4$ since $H$ does not contain $P$ and also does not 
  fully contain $L''$. In the remaining cases we assume due to symmetry that $H$ does not fully contain $L_1$. If also another triple $L_i$ with $2\le i\le t$ is not fully contained in $H$, then we 
  have $\cM(H)\le 3t+5-2\cdot 2$. Similarly, if $H$ does not fully contain $L''$, then we have $\cM(H)\le 3t+5-2\cdot 2$. So, let us assume that $H$ fully contains all triples $L_i$ for $2\le i\le t$ 
  and also fully contains $L''$. However, then there is a unique possibility for $H$ and we can easily check that $P$ is not contained in $H$ and we also have $\cM(H)\le 3t+5-2\cdot 2$.   
\end{proof} 
Note that the upper bounds in Lemma~\ref{lemma_d_4_construction} are valid for $t<2$ also while there are better constructions for dimension $k=1$ and dimension $k=3$, when $q=2$. In 
Table~\ref{table_q_2_k_6} we will see that there are $41$ $[12,6,4]_2$-codes while only the example from Lemma~\ref{lemma_d_4_construction} has locality $r=2$.
 
\begin{theorem}
  \label{thm_m_q_2_k_6_r_2}
  For each $d\in \mathbb{N}_{\ge 21}\cup\{17,18\}$ we have $n_2(6,d,2)=g_2(6,d)$. Moreover, we have $n_2(6,1,2)=n_2(6,2,2)=9$, $n_2(6,3,2)=n_2(6,4,2)=12$, $n_2(6,5,2)=n_2(6,6,2)=15$, 
  $n_2(6,7,2)=n_2(6,8,2)=18$, $n_2(6,9,2)=22$, $n_2(6,10,2)=23$, $n_2(6,11,2)=25$, $n_2(6,12,2)=26$, $n_2(6,13,2)=30$, $n_2(6,14,2)=31$, $n_2(6,15,2)=n_2(6,16,2)=33$, $n_2(6,19,2)=41$, and $n_2(6,19,2)=42$. 
\end{theorem}
\begin{proof}
  For $d\ge 21$ we can consider the Solomon-Stifler construction and use Lemma~\ref{lemma_solomon_stifler_r_2} to check that all those examples have locality $r=2$ (and in some cases even locality $r=1$). 
  Proposition~\ref{prop_reed_muller_parameters_exact} yields $n_2(6,16,2)=33$ and Corollary~\ref{cor_d_2_r_2} yields $n_2(6,1,2)=n_2(6,2,2)=9$. Lemma~\ref{lemma_d_4_construction} 
  gives $n_2(6,3,2)\le n_2(6,4,2)\le 12$.
  
  The generator matrices
  \begin{eqnarray*}
    &&
    %%\begin{pmatrix}
    %%  111110100000\\
    %%  000111010000\\
    %%  011111001000\\
    %%  111000000100\\
    %%  001110000010\\
    %%  011010000001
    %%\end{pmatrix}\!\!,\quad  
    \begin{pmatrix}
      111111100100000\\
      000111110010000\\
      011001101001000\\
      100011101000100\\
      101110010000010\\
      001110101000001
    \end{pmatrix}\!\!,\quad
    \begin{pmatrix}
      111111111110100000\\
      000001111111010000\\
      001110001111001000\\
      010110110011000100\\
      111000010111000010\\
      011011100101000001
    \end{pmatrix}\!\!,\quad
    \begin{pmatrix}
      1111111111111110100000\\
      0000000111111111010000\\
      0001111000011111001000\\
      0110011001100111000100\\
      1000101011111000000010\\
      0011100110101010000001
    \end{pmatrix}\!\!,\quad\\
    &&
    \begin{pmatrix}
      11111111111000000100000\\
      00000111111111110010000\\
      00111001111001111001000\\
      01011010001010111000100\\
      11101100111111011000010\\
      11110011011111101000001
    \end{pmatrix}\!\!,\quad 
    \begin{pmatrix}
      1111111111111110000100000\\
      0000000111111111110010000\\
      0001111000011110111001000\\
      0110011001100111011000100\\
      1010111110101011111000010\\
      0100101010111101001000001\\
    \end{pmatrix}\!\!,\quad
    \begin{pmatrix}
      11111111111111100000100000\\
      00000001111111111100010000\\
      00011110000111101110001000\\
      01100110011001101101000100\\
      10001010101000111111000010\\
      00111101001010110101000001
    \end{pmatrix}\!\!,\quad\\ 
    &&
    \begin{pmatrix}
      111111111111111111111110100000\\
      000000000001111111111111010000\\
      000001111110000001111111001000\\
      001110001110001110001111000100\\
      010110110010110010010110000010\\
      100011010110010110101010000001
    \end{pmatrix}\!\!,\quad
    \begin{pmatrix}
      1111111111111111111110000100000\\
      0000000001111111111111110010000\\
      0000011110000001111110111001000\\
      0011100110001110001111111000100\\
      0101111000110010010111001000010\\
      0110001011001110110010011000001
    \end{pmatrix}\!\!,\quad\\
    &&
    \begin{pmatrix}
      1111111111111111111000000000000100000\\
      0000000000111111111111111110000010000\\
      0000011111000011111000011111110001000\\
      0011100011001100111001100110111000100\\
      0100101101010101001010101011011000010\\
      1001100110100101010100101101101000001
    \end{pmatrix}\!\!,\quad
    \begin{pmatrix} 
      11111111111111111110000000000000100000\\
      00000000011111111111111111110000010000\\
      00001111100001111110000111111110001000\\
      00110001100110001110011001110111000100\\
      01010010101010110010101010111011000010\\
      10010111001101010111110100011101000001
    \end{pmatrix}\!\!,\quad\\ 
    &&
    \begin{pmatrix}
      11111111111111111111111000000000000100000\\
      00000000000001111111111111111111000010000\\
      00000001111110000111111000001111111001000\\
      00011110001110011000111001110011011000100\\
      01100110110010101001011010110101101000010\\
      10101010010100110010001111011001110000001
    \end{pmatrix}\!\!,\quad
    \begin{pmatrix}
      111111111111111111111110000000000000100000\\
      000000000001111111111111111111111100010000\\
      000001111110000001111110000011111110001000\\
      001110001110000110011110011100011101000100\\
      010110010010011011100110100100101111000010\\
      111010110100101110101011101101010111000001
    \end{pmatrix}  
  \end{eqnarray*}
  give %% $[12,6,4]_2$-, 
  $[15,6,6]_2$-, $[18,6,8]_2$-, $[22,6,9]_2$-, $[23,6,10]_2$-, $[25,6,11]_2$-, $[26,6,12]_2$-, $[30,6,13]_2$-, $[31,6,14]_2$-, $[37,6,17]_2$-, $[38,6,18]_2$-, $[41,6,19]_2$-, 
  and $[42,6,20]_2$-codes with locality $r=2$, so that we have %% $n_2(6,3,2)\le n_2(6,4,2)\le 12$, 
  $n_2(6,5,2)\le n_2(6,6,2)\le 15$, $n_2(6,7,2)\le n_2(6,8,2)\le 18$, $n_2(6,9,2)\le 22$, 
  $n_2(6,10,2)\le 23$, $n_2(6,11,2)\le 25$, $n_2(6,12,2)\le 26$, $n_2(6,13,2)\le 30$, $n_2(6,14,2)\le 31$, $n_2(6,17,2)\le 37$, $n_2(6,18,2)\le 38$, $n_2(6,19,2)\le 41$, and $n_2(6,20,2)\le 42$. 
  For the lower bounds we have $n_2(6,6,2)\ge n_2(6,6)=15$, $n_2(6,8,2)\ge n_2(6,8)=18$, $n_2(6,9,2)\ge n_2(6,9)=22$, $n_2(6,10,2)\ge n_2(6,10)=23$, $n_2(6,11,2)\ge n_2(6,11)=25$, $n_2(6,12,2)\ge n_2(6,12)=26$,
  $n_2(6,17,2)\ge n_2(6,17)=37$, $n_2(6,18,2)\ge n_2(6,18)=38$, $n_2(6,19,2)\ge n_2(6,19)=41$, and $n_2(6,20,2)\ge n_2(6,20)=42$. Due to length restrictions we deduce the lower bounds 
  $n_2(6,4,2)\ge n_2(6,3,2)\ge 12$, $n_2(6,5,2)\ge 15$, $n_2(6,7,2)\ge 18$, $n_2(6,13,2)\ge 30$, $n_2(6,14,2)\ge 31$, and $n_2(6,15,2)\ge 33$ from ILP computations.         

\end{proof}
We remark that our stated $[18,6,8]_2$-code is a projective two-weight code. It belongs to the family of BY codes \cite{bierbrauer1997family}. A $[26,6,12]_2$-code with locality $r=2$ 
can be obtained by shortening the unique $[27,7,12]_2$-code.\footnote{Uniqueness was e.g.\ computationally verified in \cite{bouyukliev2001optimal}. For a purely theoretic argument see 
{\lq\lq}The uniqueness of the binary linear $[27,7,12]$ code{\rq\rq} by A. E.~Brouwer from April 1992, available at https://www.win.tue.nl/$\sim$aeb/preprints.html.} There is a unique 
$[38, 6, 18]_2$-code, see e.g.\ \cite{bouyukliev2001optimal}, shortening gives $[37,6,17]_2$-codes with locality $r=2$. Note that both codes attain the Griesmer bound, see e.g.\ 
\cite{helleseth1983new} for constructions of binary codes attaining the Griesmer bound when $d>2^{k-1}$. We remark that there are $7$ non-isomorphic $[43,7,20]_2$ codes, 
see e.g.\ \cite{bouyukliev2001optimal}, and all their shortenings yield $[42,6,20]_2$-codes with locality $r=2$. 

\begin{theorem}
  \label{thm_m_q_2_k_7_r_2}
  We have $n_2(7,1,2)=n_2(7,2,2)=11$, $n_2(7,3,2)=n_2(7,4,2)=14$, $n_2(7,5,2)=17$, $n_2(7,6,2)=18$, $n_2(7,7,2)=n_2(7,8,2)=20$, $n_2(7,9,2)=24$, $n_2(7,10,2)=25$, $n_2(7,11,2)=27$, 
  $n_2(7,12,2)=28$, $n_2(7,29,2)=62$, $n_2(7,30,2)=63$, $n_2(7,31,2)=n_2(7,32,2)=65$ and $n_2(7,d,2)=n_2(7,d)$ for all other values of $d$.    
\end{theorem}
\begin{proof}
  Corollary~\ref{cor_d_2_r_2} yields $n_2(7,1,2)=n_2(7,2,2)=11$, Lemma~\ref{lemma_d_4_construction} yields $n_2(7,3,2)\le n_2(7,4,2)\le 14$, and Proposition~\ref{prop_reed_muller_parameters_exact} 
  yields $n_2(6,32,2)=65$. For $45\le d\le 64$ and for $d\ge 73$ we can consider the Solomon-Stifler construction and use Lemma~\ref{lemma_solomon_stifler_r_2} to check that all those examples 
  have locality $r=2$ (and in some cases even locality $r=1$). The Magma function $\operatorname{BKLC}(q,n,k)$ (Best Known Linear Codes)  yields for small parameters $q$, $k$, and $n$ an 
  $[n,k,d]_q$-code that maximizes the minimum distance $d$ \cite{MR1484478}. Using $q=2$, $n=n_2(7,d)$, and $k=7$ we obtained a series of codes, one for each value of $d$, that we can check for 
  locality $r\le 2$. This check was successful for $d\in\{14,\dots,16\}\cup\{20,\dots,24\}\cup\{27,28\}\cup\{33,\dots,44\}\cup\{65,\dots,72\}$.
  %% Magma online: http://magma.maths.usyd.edu.au/calc/
  %% C := BKLC(GF(2),145,7);
  %% C;  
  
  The generator matrices
  \begin{eqnarray*}
    &&
    \begin{pmatrix}
      11111110001000000\\
      00011111100100000\\
      01100110010010000\\
      10100110100001000\\
      11101111100000100\\
      10011011000000010\\
      01111111110000001
    \end{pmatrix}\!\!,\quad
    \begin{pmatrix}
      111111111001000000\\
      000000111110100000\\
      000111111010010000\\
      011001011000001000\\
      011111101110000100\\
      001111100000000010\\
      011011000010000001
    \end{pmatrix}\!\!,\quad
    \begin{pmatrix}
      11111111111001000000\\
      00000111111100100000\\
      00011001111010010000\\
      00101110011010001000\\
      00111010101100000100\\
      01101111000100000010\\
      10101001011100000001
    \end{pmatrix}\!\!,\quad\\ 
    && 
    \begin{pmatrix}
      111111111110000001000000\\
      000001111111111100100000\\
      001110000110011110010000\\
      010010011110100110001000\\
      110111100111101110000100\\
      001101101001011110000010\\
      001110011011001000000001
    \end{pmatrix}\!\!,\quad
    \begin{pmatrix}
      1111111111111100001000000\\
      0000000111111111100100000\\
      0001111000111101110010000\\
      0010011011000111110001000\\
      0110100101011100110000100\\
      0111001000101010110000010\\
      1100111000110010010000001
    \end{pmatrix}\!\!,\quad\\ 
    && 
    \begin{pmatrix}
      111111111111111000001000000\\
      000000000111111111110100000\\
      000011111000111001110010000\\
      001100111011001010110001000\\
      110001011001110111000000100\\
      110110011000011010010000010\\
      011110101110110011110000001
    \end{pmatrix}\!\!,\quad
    \begin{pmatrix}
      1111111111111110000001000000\\
      0000000001111111111100100000\\
      0000111110001110011110010000\\
      0111000110010110101110001000\\
      1011011001100011110000000100\\
      0101101011100100110100000010\\
      1001001110111000111000000001
    \end{pmatrix}\!\!,\quad\\
    && 
    \begin{pmatrix} 
      1111111111111110000000001000000\\
      0000000111111111111111000100000\\
      0000111000111110000111110010000\\
      0011011011000110011011010001000\\
      0111101101001010101100110000100\\
      1001011101011000011101000000010\\
      0010111011010011100101000000001
    \end{pmatrix}\!\!,\quad
    \begin{pmatrix}
      111111111111111111100000000000001000000\\
      000000000111111111111111111100000100000\\
      000011111000001111100001111111100010000\\
      001100111001110001100110011101110001000\\
      010101011010010110101110100110110000100\\
      011110001000111011011000101110000000010\\
      101010011100010101111011000110100000001
    \end{pmatrix}\!\!,\quad\\
 \end{eqnarray*} 
 \begin{eqnarray*}
    && 
    \begin{pmatrix}
      1111111111111111111000000000000001000000\\
      0000000001111111111111111111000000100000\\
      0000111110000011111000011111111100010000\\
      0011001110011100011001100011001110001000\\
      0101010110101101100010100101010110000100\\
      1001101010010100111111001000110010000010\\
      1100001011001111001101010000101110000001
    \end{pmatrix}\!\!,\quad
    \begin{pmatrix}
      111111111111111111111100000000000001000000\\
      000000000000111111111111111111100000100000\\
      000000111111000011111100000111111100010000\\
      000111001111001100001100011001101110001000\\
      011011110011011101110111101011111110000100\\
      101111010101100100111101110110110110000010\\
      110101100011001110011010101100110100000001    
    \end{pmatrix}\!\!,\quad\\
  \end{eqnarray*}
  give $n_2(7,5,2)\le 17$, $n_2(7,6,2)\le 18$, $n_2(7,7,2)\le n_2(7,8,2)\le 20$, $n_2(7,9,2)\le 24$, $n_2(7,10,2)=25$, $n_2(7,11,2) \le 27$, $n_2(7,12,2)\le 28$, $n_2(7,13,2)\le 31$, 
  $n_2(7,17,2)\le 39$, $n_2(7,18,2)\le 40$, and $n_2(7,19,2)\le 42$. For the other cases we will state explicit examples by listing sets of points for the occurring non-zero multiplicities, 
  where the points are stated as integers whose base-2-representation is a generator of the point. 
  \begin{itemize}
    \item $d=55$: 1$\to\{1,2,4,7,8,9,12,13,16,19,21,22,25,27,32,35,37,38,39,41,44,47,49,50,55,56,59,61,64,$ $67,73,76,81,84,90,91,94,98,100,103,104,105,107,108,110,112,115,117,\dots,122\}$, 2$\to\{70\}$;
    \item $d=56$: 1$\to\{1,2,4,8,10,\dots,13,16,19,21,22,25,31,32,35,37,38,40,41,42,44,47,50,52,55,56,59,61,$ $62,64,67,69,74,76,80,82,84,87,91,93,94,98,100,103,110,111,112,115,117,121,122,125\}$, 3$\to\{97\}$;
    \item $d=62$: 1$\to\{1,2,4,5,8,\dots,11,13,14,16,\dots,19,21,22,28,30,32,35,36,37,40,43,44,46,49,51,52,55,$ $59,60,61,64,67,69,72,74,76,79,81,82,85,87,90,\dots,93,98,99,103,105,106,112,115,120,122,125,$\\ $126,127\}$, 2$\to\{101\}$;
    \item $d=63$: 1$\to\{1,3,4,7,8,11,13,15,16,17,20,21,26,27,29,30,32,33,40,42,45,46,49,50,54,55,56,59,$\\ $60,64,65,68,73,74,76,77,81,83,85,86,88,91,93,95,98,99,101,102,106,107,108,111,114,115,120,$\\ $125,126\}$, 2$\to\{2,36,116\}$;
    \item $d=65$: 1$\to\{1,2,4,7,8,11,13,14,16,19,21,22,25,26,28,31,32,35,37,38,41,42,44,47,49,50,52,55,$\\ $56,59,61,62,64,67,69,70,73,74,76,79,81,82,84,87,88,91,93,94,96,97,98,100,103,104,107,109,$\\ $110,112,115,117,118,121,122,124,127\}$.
  \end{itemize}
  The lower bounds in those cases where $n_2(7,d,2)>n_2(7,d)$ are verified using ILP computations.
\end{proof}

\begin{lemma}
  \label{lemma_constructions_r_1}
  We have $n_q(k,2d',1)\le 2\cdot n_q(k,d')$, $n_q(k,d'+2\cdot q^{k-1},1)\le 2\cdot \left(q^k-1\right)/(q-1)+n_q(k,d')$, and $n_q(k,3d'-1,1)\le 3\cdot n_q(k,d')-1$. 
\end{lemma}
\begin{proof}
  If $\cM$ is the multiset of points corresponding to an $[n',k,d']_q$-code, then $2\cdot \cM$ corresponds to an $[2n',k,2d']_q$-code with locality $1$ since each point has an 
  even multiplicity. Denoting the ambient space by $A$, we have that $\cM+2\cdot \chi_A$ corresponds to an $\left[2\cdot \left(q^k-1\right)/(q-1)+n',k,2\cdot q^{k-1}+d'\right]_q$-code 
  with locality $r=1$ since each point has multiplicity at least $2$.  Observing that $3\cdot \cM$ corresponds to an $[3n',k,3d']_q$-code with locality $1$ since each point has an 
  multiplicity that is divisible, we can decrease the point multiplicity of one point with positive multiplicity by one to obtain an $[3n'-1,k,3d'-1]_q$-code.
\end{proof}

\begin{theorem}
  \label{thm_m_q_2_k_3_r_1}
  We have $n_2(3,1,1)=n_2(3,2,1)=6$, $n_2(3,3,1)=n_2(3,4,1)=8$, $n_2(3,5,1)=11$, $n_2(3,6,1)=12$, $n_2(3,7,1)=14$, and $n_2(3,d,1)=n_2(3,d)$ for all $d\ge 8$. 
\end{theorem}
\begin{proof}
  The lower bounds for $d\le 7$ can be obtained by ILP computations. The necessary upper bounds can be deduced from Lemma~\ref{lemma_constructions_r_1} and $n_2(k,d-1,1)\le n_2(k,d,1)$.  
\end{proof}

\begin{theorem}
  \label{thm_m_q_2_k_4_r_1}
  We have $n_2(4,1,1)=n_2(4,2,1)=8$, $n_2(4,3,1)=n_2(4,4,1)=10$, $n_2(4,5,1)=n_2(4,6,1)=14$, $n_2(4,7,1)=n_2(4,8,1)=16$, $n_2(4,9,1)=20$, $n_2(4,10,1)=22$, $n_2(4,11,1)=23$, $n_2(4,12,1)=24$, 
  $n_2(4,13,1)=27$, $n_2(4,14,1)=28$, $n_2(4,15,1)=30$, $n_2(4,d,1)=n_2(4,d)$ for all $d\ge 16$. 
\end{theorem}
\begin{proof}
  The lower bounds for $d\le 15$ can be obtained by ILP computations. Except $d=9$ and $d=13$, the necessary upper bounds can be deduced from Lemma~\ref{lemma_constructions_r_1} and 
  $n_2(k,d-1,1)\le n_2(k,d,1)$. For $d=9$ let $A$ denote the ambient space and $E:=\left\langle e_1,e_2,e_3\right\rangle$. A suitable $[19,4,9]_2$-code with locality $r=1$ can be 
  obtained from the multiset of points $2\cdot\chi_{A\backslash E} +\chi_{\left\langle e_4\right\rangle}  +\chi_{\left\langle e_4+e_1\right\rangle}+\chi_{\left\langle e_4+e_2\right\rangle} 
  +\chi_{\left\langle e_4+e_3\right\rangle}$. For $d=13$ let $A$ denote the ambient space and $E:=\left\langle e_1,e_2\right\rangle$. A suitable $[27,4,13]_2$-code with locality $r=1$ can be 
  obtained from the multiset of points $2\cdot\chi_{A\backslash L} +\chi_{\left\langle e_3\right\rangle}  +\chi_{\left\langle e_3+e_1\right\rangle}+\chi_{\left\langle e_3+e_2\right\rangle}$.
\end{proof}

\begin{theorem}
  \label{thm_m_q_2_k_5_r_1}
  We have 
  \begin{center}
    \begin{tabular}{cccccccccccccccccc}
      \hline
      $d$          &  1 &  2 &  3 &  4 &  5 &  6 &  7 &  8 &  9 & 10 & 11 & 12 & 13 & 14 & 15 & 16 & 17 \\ 
      $n_2(5,d,1)$ & 10 & 10 & 12 & 12 & 17 & 18 & 19 & 20 & 24 & 25 & 27 & 28 & 30 & 30 & 32 & 32 & 37 \\
      \hline
    \end{tabular}
    
    \smallskip
    
    \begin{tabular}{ccccccccccccccccc}
      \hline
      $d$          & 18 & 19 & 20 & 21 & 22 & 23 & 24 & 25 & 26 & 27 & 28 & 29 & 30 & 31 & 36 & 38 \\
      $n_2(5,d,1)$ & 38 & 41 & 42 & 45 & 46 & 47 & 48 & 52 & 54 & 55 & 56 & 59 & 60 & 62 & 72 & 76\\
      \hline
    \end{tabular}
   \end{center}
   for small values of $d$ and $n_2(5,d,1)=n_2(5,d)$ in all other cases.
\end{theorem}
\begin{proof}
  The lower bounds for $d\le 38$ can be obtained by ILP computations. The necessary upper bounds can almost always be deduced from the constructions in Lemma~\ref{lemma_constructions_r_1} 
  and $n_2(k,d-1,1)\le n_2(k,d,1)$. For the other cases we will state explicit examples by listing sets of points for the occurring non-zero multiplicities, where the points are stated as
  integers whose base-2-representation is a generator of the point. 
  \begin{itemize}
    \item $d=5$: 2$\to\{1,2 , 8, 16,23,29,30\}$, 3$\to\{4\}$;
    \item $d=7$: 2$\to\{2,4,8,14,16,23,26,29\}$, 3$\to\{1\}$;    
    \item $d=9$: 2$\to\{1,2,15,16,23,27\}$, 3$\to\{4,8,29,30\}$;    
    \item $d=10$: 2$\to\{1,2,4,8,16\}$, 3$\to\{15,23,27,29,30\}$;
    \item $d=11$: 2$\to\{1,2,4,5,8,15,16,17,27,28,29,30\}$, 3$\to\{23\}$;
    \item $d=17$: 2$\to\{2,4,7,8,11,14,16,21,22,25,31\}$, 3$\to\{1,13,19,26,28\}$;
    \item $d=18$: 2$\to\{7,11,13,14,19,21,22,25,26,28\}$, 3$\to\{1,2,4,8,16,31\}$;
    \item $d=19$: 2$\to\{1,2,4,8,11,16,28\}$, 3$\to\{7,13,14,19,21,22,25,26,31\}$;
    \item $d=21$: 2$\to\{8,13,31\}$, 3$\to\{1,2,4,7,11,14,16,19,21,22,25,26,28\}$;
    \item $d=23$: 2$\to\{8\}$, 3$\to\{1,2,4,7,11,13,14,16,19,21,22,25,26,28,31\}$;
    \item $d=25$: 2$\to\{0,1,2,3,4,6,7,8,10,13,14,15,16,21,22,25,26,27,28,30\}$, 3$\to\{13,19,21,25\}$;
    \item $d=27$: 2$\to\{1,2,4,7,8,11,13,14,16,18,19,21,22,25,26,28,31\}$, 3$\to\{3,5,9,15,20,24,30\}$;
    \item $d=29$: 2$\to\{1,2,3,4,5,7,8,10,11,12,13,14,16,17,18,20,22,23,24,25,27,28,29,30,31\}$,\\ 3$\to\{6,9,19\}$;
    \item $d=35$: 2$\to\{3,5,6,7,9,10,11,12,13,14,17,18,19,20,21,22,24,25,26,28\}$,\\ 3$\to\{1,2,4,8,15,16,23,27,29,30\}$;
    \item $d=37$: 2$\to\{1,4,5,6,8,9,11,12,18,19,22,23,26,27,30,31\}$,\\ 3$\to\{2,3,7,10,14,15,16,17,20,21,24,25,28,29\}$.
  \end{itemize}
\end{proof}

\begin{theorem}
  \label{thm_m_q_2_k_6_r_1}
  We have 
  \begin{center}
    \begin{tabular}{cccccccccccccccccc}
      \hline
      $d$          &  1 &  2 &  3 &  4 &  5 &  6 &  7 &  8 &  9 & 10 & 11 & 12 & 13 & 14 & 15 & 16 & 17 \\ 
      $n_2(6,d,1)$ & 12 & 12 & 14 & 14 & 20 & 20 & 22 & 22 & 27 & 28 & 30 & 30 & 33 & 34 & 35 & 36 & 41 \\
      \hline
    \end{tabular}
    
    \smallskip
    
    \begin{tabular}{cccccccccccccccccc}
      \hline
      $d$          & 18 & 19 & 20 & 21 & 22 & 23 & 24 & 25 & 26 & 27 & 28 & 29 & 30 & 31 & 32 & 33 & 34 \\
      $n_2(6,d,1)$ & 42 & 45 & 46 & 49 & 50 & 51 & 52 & 57 & 58 & 59 & 60 & 62 & 62 & 64 & 64 & 70 & 72 \\
      \hline
    \end{tabular}
    
    \smallskip
    
    \begin{tabular}{cccccccccccccccccc}
      \hline
      $d$          & 35 & 36 & 37 & 38 & 39 & 40 & 41 & 42 & 43 & 44 & 45 & 46 & 47 & 48 &  49 &  50 &  51 \\
      $n_2(6,d,1)$ & 74 & 76 & 78 & 80 & 81 & 82 & 86 & 88 & 90 & 90 & 93 & 94 & 95 & 96 & 101 & 102 & 105 \\
      \hline
    \end{tabular}
    
    \smallskip
    
    \begin{tabular}{cccccccccccccccccc}
      \hline
      $d$          &  52 &  53 &  54 &  55 &  56 &  57 &  58 &  59 &  60 &  61 &  62 &  63 &  64 &  71 \\
      $n_2(6,d,1)$ & 106 & 109 & 110 & 111 & 112 & 116 & 118 & 119 & 120 & 123 & 124 & 126 & 126 & 142 \\
      \hline
    \end{tabular}
    
    \smallskip
    
    \begin{tabular}{cccccccccccccccccc}
      \hline
      $d$          &  72 &  75 &  76 &  78 &  84 \\
      $n_2(6,d,1)$ & 143 & 150 & 151 & 156 & 168 \\
      \hline
    \end{tabular}
   \end{center}
   for small values of $d$ and $n_2(6,d,1)=n_2(6,d)$ in all other cases.
\end{theorem}
\begin{proof}
  The lower bounds for $d\le 84$ can be obtained by ILP computations. The necessary upper bounds can almost always be deduced from the constructions in Lemma~\ref{lemma_constructions_r_1} 
  and $n_2(k,d-1,1)\le n_2(k,d,1)$. For the other cases we will state explicit examples by listing sets of points for the occurring non-zero multiplicities, where the points are stated as
  integers whose base-2-representation is a generator of the point. 
  \begin{itemize}
    \item $d=9$: 2$\to\{8,13,16,42,52,55\}$, 3$\to\{1,2,4,27,32\}$; 
    \item $d=13$: 2$\to\{1,2,4,7,8,16,19,32,37,41,47,50,52,56,61\}$, 3$\to\{29\}$;
    \item $d=15$: 2$\to\{1,2,4,8,23,27,29,30,32,47,51,53,54,57,58,60\}$, 3$\to\{16\}$;
    \item $d=17$: 2$\to\{4,5,15,16,18,25,30,32,41,42,53,61,62\}$, 3$\to\{1,2,8,38,51\}$;
    \item $d=18$: 2$\to\{1,4,8,16,32,34,39,43,44,49,52,58\}$, 3$\to\{2,31,45,55,57,62\}$;
    \item $d=19$: 2$\to\{1,2,4,8,16,23,27,32,33,34,36,39,40,43,46,47,48,51,54,58,62\}$, 3$\to\{61\}$;
    \item $d=21$: 2$\to\{1,2,4,8,11,16,17,19,22,42,46,50,51,53,55,57,62\}$, 3$\to\{15,32,41,44,52\}$;
    \item $d=23$: 2$\to\{1,2,3,4,7,8,9,15,16,25,27,31,32,35,37,40,45,46,50,51,52,57,62,63\}$, 3$\to\{21\}$;
    \item $d=25$: 2$\to\{4,8,11,13,14,16,19,21,22,32,37,41,42,44,47,49,50,52,55,59,62\}$, 3$\to\{1,2,25,35,61\}$;
    \item $d=27$: 2$\to\{1,4,7,8,13,14,16,19,21,22,25,26,28,31,32,35,37,38,41,42,44,47,49,50,55,56,59,62\}$, 3$\to\{2\}$;
    \item $d=33$: 2$\to\{1,2,4,7,8,11,16,19,22,25,26,28,31,32,35,37,38,41,42,44,47,50,52,55,56,62\}$,\\ 3$\to\{13,14,21,49,59,61\}$;
    \item $d=34$: 2$\to\{1,4,7,13,14,16,19,21,22,28,31,32,41,44,47,49,50,56,59,61,62\}$,\\ 3$\to\{2,8,25,26,35,37,38,42,52,55\}$;
    \item $d=35$: 2$\to\{4,7,14,16,21,22,25,26,28,31,35,37,38,41,42,47,50,55,56,59,61,62\}$,\\ 3$\to\{1,2,8,11,13,19,32,44,49,52\}$;
    \item $d=37$: 2$\to\{1,11,16,19,21,26,28,32,35,38,42,47,49,50,52,55,61,62\}$,\\ 3$\to\{2,4,7,8,13,14,22,25,31,37,41,44,56,59\}$;
    \item $d=38$: 2$\to\{2,4,8,11,16,22,37,41,42,47,50,52,55,56,61,62\}$,\\ 3$\to\{1,7,13,14,19,21,25,26,28,31,32,35,37,38,44,49,59\}$;
    \item $d=39$: 2$\to\{1,4,7,8,11,13,14,19,21,23,37,38,41,44,52,55,56,61\}$,\\ 3$\to\{2,16,22,25,26,28,31,32,35,42,47,49,50,59,62\}$;
    \item $d=40$: 2$\to\{8,14,16,19,26,31,32,35,42,47,49,50,52,54,55,59,61\}$,\\ 3$\to\{1,2,4,7,11,13,21,22,25,28,37,38,41,44,56,62\}$;
    \item $d=41$: 2$\to\{2,21,22,38,47,49,55,59,61,62\}$, 3$\to\{1,4,7,8,11,13,14,16,19,25,26,28,31,$\\ $32,35,37,41,42,44,50,52,56\}$;
    \item $d=45$: 2$\to\{1,2,4,7,8,9,10,11,13,14,15,16,21,22,24,25,26,27,28,29,30,32,35,37,38,$\\ $41,42,43,44,45,46,47,49,50,52,56,57,58,60,61,62,63\}$, 3$\to\{12,31,55\}$;
    \item $d=47$: 2$\to\{8\}$, 3$\to\{1,2,4,7,11,13,14,16,19,21,22,25,26,28,31,32,35,37,38,41,42,44,$\\ $47,49,50,52,55,56,59,61,62\}$;
    \item $d=49$: 2$\to\{1,3,4,5,6,8,10,13,14,15,16,17,18,21,22,23,25,27,28,30,32,33,35,36,38,$\\ $39,40,42,43,44,45,47,48,50,51,52,53,55,57,59,60,62,63\}$, 3$\to\{2,9,26,29,56\}$;
    \item $d=50$: 2$\to\{1,3,4,6,8,10,13,14,15,16,18,19,21,23,25,26,27,28,30,31,32,34,35,37,38,$\\ $39,41,42,43,44,46,49,50,51,52,54,56,58,59,61,62,63\}$, 3$\to\{2,7,11,22,47,55\}$;
    \item $d=51$: 2$\to\{1,4,5,6,7,8,12,13,14,15,16,17,18,19,22,24,25,26,27,28,32,33,34,35,40,$\\ $41,42,43,51,52,53,54,55,57,58,60,61,62,63\}$, 3$\to\{2,11,21,31,36,39,45,46,48\}$;
    \item $d=53$: 2$\to\{1,3,4,6,8,9,11,12,13,14,15,16,18,20,21,23,24,26,29,30,31,32,33,35,36,$\\ $37,38,39,40,41,43,44,46,48,49,50,51,52,53,55,56,57,58,59,61,62,63\}$, 3$\to\{2,5,19,25,47\}$;
    \item $d=55$: 2$\to\{1,3,4,6,8,10,13,15,16,18,21,23,25,26,27,28,30,33,35,36,38,40,42,45,47,$\\ $48,50,53,55,57,59,60,62\}$, 3$\to\{2,7,11,14,19,22,31,32,37,41,44,49,52,56,61\}$;
    \item $d=57$: 2$\to\{1,3,\dots,8,10,\dots,19,21,23,\dots,28,30,\dots,37,39,\dots,44,46,47,49,\dots,56,58,$\\ $60,61,62,63\}$, 3$\to\{2,22,38,59\}$;
    \item $d=59$: 2$\to\{1,3,\dot,11,13,15,16,18,20,\dots,28,30,32,\dots,36,38,40,42,44,\dots,51,53,55,56,$\\ $57,59,\dots,,63\}$, 3$\to\{2,14,19,31,37,41,52\}$;
    \item $d=61$: 2$\to\{1,3,\dots,23,25,27,\dots,39,41,\dots,47,49,51,\dots,63\}$, 3$\to\{2,26,50\}$;
    \item $d=65$: 2$\to\{1,3,4,6,\dots,14,16,\dots,20,22,25,\dots,36,38,39,41,43,44,45,48,49,51,\dots,60,62\}$,\\ 3$\to\{2,5,15,21,24,37,40,46,47,50,56,63\}$;
    \item $d=66$: 2$\to\{1,3,4,7,\dots,12,14,\dots,19,21,\dots,37,39,\dots,42,44,47,\dots,60,62,63\}$,\\ 3$\to\{2,5,6,13,20,38,45,46,61\}$;
    \item $d=67$: 2$\to\{1,3,4,6,8,9,10,12,13,15,16,17,19,21,22,24,\dots,32,35,37,38,40,\dots,47,49,$\\ $51,52,54,56,57,58,60,61,63\}$, 3$\to\{2,7,11,14,18,20,23,33,34,36,39,50,55,59,62\}$;
    \item $d=68$: 2$\to\{3,6,10,15,16,17,19,\dots,22,24,25,26,28,29,31,33,\dots,36,38,39,40,42,43,45,46,$\\ $47,49,52,56,61\}$, 3$\to\{1,2,4,7,8,11,13,14,18,23,27,30,32,37,41,44,48,51,53,54,57,58,60,63\}$;
    \item $d=69$: 2$\to\{1,5,\dots,10,12,13,15,\dots,18,20,21,23,25,29,\dots,35,37,40,41,43,\dots,46,48,49,$\\ $51,\dots,54,57,58,59,61\}$, 3$\to\{2,3,4,11,14,19,22,26,27,28,38,39,42,47,50,55,56,62,63\}$;
    \item $d=70$: 2$\to\{1,3,5,6,8,11,12,14,16,19,21,23,24,26,29,30,32,33,36,\dots,45,48,\dots,51,54,55,$\\ $58,\dots,63\}$, 3$\to\{2,4,7,9,10,15,17,18,20,25,28,31,34,35,46,47,52,53,56,57\}$;
    \item $d=73$: 2$\to\{1,4,8,9,12,13,16,17,20,21,24,25,26,28,29,31,\dots,34,36,37,39,\dots,42,44,45,48,$\\ $49,52,53,57,60\}$, 3$\to\{2,3,6,7,10,11,14,15,18,19,22,23,27,30,35,38,43,46,47,50,51,54,55,$\\ $58,59,62,63\}$;
    \item $d=74$: 2$\to\{1,4,5,8,9,12,13,15,16,17,18,20,21,24,25,28,32,33,34,36,37,40,41,44,49,52,53,$\\ $56,57,60,61,63\}$, 3$\to\{2,3,6,7,10,11,14,19,22,23,26,27,30,31,35,38,39,42,43,46,47,50,51,$\\ $54,55,58,59,62\}$;
    \item $d=77$: 2$\to\{1,4,5,8,9,12,13,16,17,20,21,24,25,28,29,30,34,35,38,39,42,43,44,46,47,51,54,$\\ $55,58,59,62,63\}$, 3$\to\{2,3,6,7,10,11,14,15,18,19,22,23,26,27,31,32,33,36,37,40,41,45,48,$\\ $49,52,53,56,57,60,61\}$;
    \item $d=83$: 2$\to\{1,8,9,12,18,26,27,31,32,33,36,40,44,45,50,51,55,59,62,63\}$, 3$\to\{2,\dots,7,10,11,$\\ $13,\dots,17,20,\dots,25,28,29,30,34,35,37,38,39,41,42,43,46,\dots,49,52,53,54,56,57,58,60,61\}$;           
  \end{itemize}
\end{proof}

\begin{theorem}
  \label{thm_m_q_3_k_3_r_1}
  We have 
  \begin{center}
    \begin{tabular}{cccccccccccccccccc}
      \hline
      $d$          & 1 & 2 & 3 & 4 &  5 &  6 &  7 &  8 &  9 & 10 & 11 & 12 & 13 & 14 & 15 & 17 & 21 \\ 
      $n_3(3,d,1)$ & 6 & 6 & 8 & 8 & 11 & 12 & 14 & 14 & 16 & 16 & 18 & 18 & 21 & 22 & 24 & 26 & 32 \\
      \hline
    \end{tabular}
   \end{center}
   for small values of $d$ and $n_3(3,d,1)=n_3(3,d)$ in all other cases.
\end{theorem}
\begin{proof}
  The lower bounds for $d\le 21$ can be obtained by ILP computations. The necessary upper bounds can almost always be deduced from the constructions in Lemma~\ref{lemma_constructions_r_1} 
  and $n_3(k,d-1,1)\le n_3(k,d,1)$. For the other cases we will state explicit examples by listing sets of points for the occurring non-zero multiplicities, where the points are stated as
  integers denoting the position in the list of lexicographical minimal vectors that are generators of a point (starting to count from zero). 
  \begin{itemize}
    \item $d=5$: 2$\to\{0,1,11,12\}$, 3$\to\{4\}$;
    \item $d=13$: 2$\to\{4,5,7,11,12\}$, 3$\to\{0,1,2,9\}$.    
  \end{itemize}
\end{proof}

\begin{theorem}
  \label{thm_m_q_3_k_4_r_1}
  We have 
  \begin{center}
    \begin{tabular}{ccccccccccccccccccc}
      \hline
      $d$          & 1 & 2 &  3 &  4 &  5 &  6 &  7 &  8 &  9 & 10 & 11 & 12 & 13 & 14 & 15 & 16 & 17 & 18 \\ 
      $n_3(4,d,1)$ & 8 & 8 & 10 & 10 & 14 & 14 & 16 & 16 & 18 & 18 & 20 & 20 & 24 & 25 & 27 & 28 & 29 & 30 \\
      \hline
    \end{tabular}
    
    \smallskip    
    
    \begin{tabular}{ccccccccccccccccccc}
      \hline
      $d$          & 19 & 20 & 21 & 22 & 23 & 24 & 25 & 26 & 27 & 28 & 29 & 30 & 31 & 32 & 33 & 34 & 35 & 37 \\ 
      $n_3(4,d,1)$ & 33 & 34 & 36 & 36 & 38 & 38 & 41 & 42 & 44 & 46 & 47 & 48 & 50 & 50 & 52 & 52 & 54 & 58 \\
      \hline
    \end{tabular}    
    
    \smallskip    
    
    \begin{tabular}{ccccccccccccccccccc}
      \hline
      $d$          & 38 & 39 & 40 & 41 & 42 & 43 & 44 & 45 & 47 & 49 & 50 & 51 & 53 & 57 & 61 & 62 & 63 &  69 \\ 
      $n_3(4,d,1)$ & 59 & 61 & 62 & 63 & 64 & 67 & 68 & 69 & 72 & 75 & 76 & 78 & 80 & 87 & 93 & 94 & 95 & 104 \\
      \hline
    \end{tabular}    
        
   \end{center}
   for small values of $d$ and $n_3(4,d,1)=n_3(4,d)$ in all other cases.
\end{theorem}
\begin{proof}
  The lower bounds for $d\le 69$ can be obtained by ILP computations. The necessary upper bounds can almost always be deduced from the constructions in Lemma~\ref{lemma_constructions_r_1} 
  and $n_3(k,d-1,1)\le n_3(k,d,1)$. For the other cases we will state explicit examples by listing sets of points for the occurring non-zero multiplicities, where the points are stated as
  integers denoting the position in the list of lexicographical minimal vectors that are generators of a point (starting to count from zero). 
  \begin{itemize}
    \item $d=13$: 2$\to\{0,1,9,13,26,33\}$, 3$\to\{4,20,30,34\}$;
    \item $d=14$: 2$\to\{1,13,21,26,38\}$, 3$\to\{0,4,11,24,34\}$;
    \item $d=15$: 2$\to\{0,1,2,4,11,12,13,18,21,24,27,32\}$, 3$\to\{34\}$;
    \item $d=17$: 2$\to\{0\}$, 3$\to\{1,4,11,13,18,26,30,32,37\}$;
    \item $d=19$: 2$\to\{0,1,4,5,7,13,16,27,29,35,36,39\}$, 3$\to\{2,14,22\}$;
    \item $d=25$: 2$\to\{0,1,4,5,9,11,13,20,21,24,25,26,30,33,35,38\}$, 3$\to\{2,14,36\}$;
    \item $d=26$: 2$\to\{1,3,4,5,7,9,13,15,16,24,28,29,35,36,39\}$, 3$\to\{0,20,25,32\}$;
    \item $d=27$: 2$\to\{0,1,2,3,4,5,6,13,16,19,22,26,30,31,36,38\}$, 3$\to\{10,14,27,39\}$;
    \item $d=29$: 2$\to\{0,1,2,4,6,9,10,13,15,16,17,18,19,22,23,24,25,26,29,32,35,37\}$, 3$\to\{8\}$;
    \item $d=37$: 2$\to\{1,3,4,6,7,8,11,12,13,14,18,19,21,23,24,25,27,28,31,33,34,35,39\}$, 3$\to\{0,17,29,38\}$;
    \item $d=38$: 2$\to\{0,1,3,4,6,7,8,11,13,15,16,20,21,22,23,26,27,30,32,33,36,38\}$, 3$\to\{12,17,28,34,37\}$;
    \item $d=39$: 2$\to\{0,1,3,4,5,8,9,12,13,16,17,21,23,24,31,32,35,36,37,39\}$, 3$\to\{10,15,20,25,27,28,29\}$;
    \item $d=41$: 2$\to\{0,1,2,4,12,13,14,18,20,21,22,24,26,27,28,32,33,39\}$, 3$\to\{5,7,9,11,16,29,34,35,37\}$;
    \item $d=43$: 2$\to\{0,1,4,5,8,9,10,12,17,29,30,33,34,36\}$, 3$\to\{3,13,14,18,19,21,22,24,25,26,32,37,38\}$;
    \item $d=45$: 2$\to\{1,4,10,12,15,20,23,25,27,28,29,31,36\}$, 3$\to\{0,3,8,9,13,16,17,21,24,32,35,37,39\}$, 4$\to\{5\}$;
    \item $d=49$: 2$\to\{1,\dots,11,13,\dots,17,19,\dots,22,25,\dots,36,38\}$, 3$\to\{0,24,39\}$;
    \item $d=67$: 2$\to\{0,5,7,12,14,\dots,17,19,21,22,27,29,32,34,39\}$, 3$\to\{1,2,4,6,8,9,10,11,13,18,20,$\\ $23,\dots,26,28,30,31,33,35,\dots,38\}$;
    \item $d=68$: 2$\to\{0,1,3,4,8,12,31,\dots,39\}$, 3$\to\{5,6,7,9,10,11,13,\dots,30\}$;
  \end{itemize}
\end{proof}
We remark that all $[69,4,45]_3$-codes that have locality $r=1$, when considered as a multiset of points in $\PG(3,3)$, have maximum point multiplicity of at least $4$ (and not $3$ as in all
other stated examples for the other parameters). 

\section{Enumeration results}
\label{sec_enumeration}
Instead of solving ILPs to determine locally recoverable codes with the minimum possible length $n_q(k,d,r)$ one may also enumerate all $[n,k,d]_q$-codes with minimum possible length $n=n_q(k,d)$ 
or slightly larger length and check whether those codes have relatively small locality constants $r$. For those enumerations we have applied the software \texttt{LinCode} \cite{bouyukliev2021computer}. 
We present our computational results in Tables~\ref{table_q_2_k_5}, \ref{table_q_2_k_6}, and \ref{table_q_2_k_7}. Locality $r=1$ did not occur in any of these cases.

\begin{table}
  \begin{center}
    \begin{tabular}{l|rrrrrrrrrrrrrrrrrr}
      \hline
      $n$         &  8 & 9 & 10 & 13 & 14 & 15 & 16 & 17 & 20 & 21 & 23 & 24 & 27 & 28 & 30 & 31 & 36 & 37 \\
      $d$         &  2 & 3 &  4 &  5 &  6 &  7 &  8 &  8 &  9 & 10 & 11 & 12 & 13 & 14 & 15 & 16 & 17 & 18 \\
      \hline
      %% $\# r=1$    &  0 & 0 &  0 &  0 &  0 &  0 &  0 &  0 &  0 &  0 &  0 &  0 &  0 &  0 &  0 &  0 &  0 &  0 \\ 
      $\# r=2$    &  2 & 0 &  0 &  5 &  3 &  0 &  0 &  3 &  3 &  2 &  1 &  1 &  1 &  1 &  1 &  1 & 70 & 13 \\ 
      $\# r\ge 3$ & 25 & 5 &  4 & 10 &  3 &  1 &  1 &  1 &  0 &  0 &  0 &  0 &  0 &  0 &  0 &  0 &  0 &  1 \\
      \hline
      $\#$        & 27 & 5 &  4 & 15 &  6 &  1 &  1 &  4 &  3 &  2 &  1 &  1 &  1 &  1 &  1 &  1 & 70 & 14 \\ 
      \hline  
    \end{tabular}
    
    \smallskip
    
    \begin{tabular}{l|rrrrrrrrrrrrrrrrrr}
      \hline
      $n$         & 39 & 40 & 43 & 44 & 46 & 47 & 51 & 52 & 54 & 55 \\
      $d$         & 19 & 20 & 21 & 22 & 23 & 24 & 25 & 26 & 27 & 28 \\
      \hline
      %% $\# r=1$    &  0 &  0 &  0 &  0 &  0 &  0 &  0 &  0 &  0 &  0 \\ 
      $\# r=2$    &  7 &  3 &  5 &  2 &  2 &  1 & 12 &  4 &  2 &  1 \\ 
      $\# r\ge 3$ &  0 &  0 &  0 &  0 &  0 &  0 &  0 &  0 &  0 &  0 \\
      \hline
      $\#$        &  7 &  3 &  5 &  2 &  2 &  1 & 12 &  4 &  2 &  1 \\ 
      \hline  
    \end{tabular}
    \caption{Locality constants for $5$-dimensional binary codes with small lengths.}
    \label{table_q_2_k_5}
  \end{center}
\end{table}

\begin{table}
  \begin{center}
    \begin{tabular}{l|rrrrrrrrrrrrrrrrrrrrrr}
      \hline
      $n$         & 7 & 10 & 11 & 12 & 14 & 15 & 17 & 18 &  22 & 23 & 25 & 26 & 29 &    30 & 30 &  31 & 31 & 32 \\ %% & 37 & 38 &  41 & 42 \\
      $d$         & 2 &  3 &  4 &  4 &  5 &  6 &  7 &  8 &   9 & 10 & 11 & 12 & 13 &    13 & 14 &  14 & 15 & 16 \\ %% & 17 & 18 &  19 & 20 \\
      \hline
      %% $\# r=1$    & 0 &  0 &  0 &  0 &  0 &  0 &  0 &  0 &   0 &  0 &  0 &  0 &  0 &     0 &  0 &   0 &  0 &  0 \\ %% &  0 &  0 &   0 &  0 \\ 
      $\# r=2$    & 0 &  0 &  0 &  1 &  0 &  1 &  0 &  1 &  93 & 12 &  4 &  2 &  0 &  4786 &  0 & 295 &  0 &  0 \\ %% &  2 &  1 & 235 & 35 \\ 
      $\# r\ge 3$ & 1 &  4 &  2 & 40 & 11 &  4 &  3 &  1 & 155 & 17 &  1 &  0 &  2 &   133 &  2 &  18 &  1 &  1 \\ %% &  0 &  0 &   0 &  0 \\
      \hline
      $\#$        & 1 &  4 &  2 & 41 & 11 &  5 &  3 &  2 &  248 & 29 &  5 &  2 &  2 & 4919 &  2 & 313 &  1 &  1 \\ %% &  2 &  1 & 235 & 35 \\ 
      \hline  
    \end{tabular}
    
    \smallskip
    
    \begin{tabular}{l|rrrrrrrrrrrrrrrrrrrrrr}    
     \hline
      $n$         & 37 & 38 &  41 & 42 \\
      $d$         & 17 & 18 &  19 & 20 \\
      \hline
      %% $\# r=1$    &  0 &  0 &   0 &  0 \\ 
      $\# r=2$    &  2 &  1 & 235 & 35 \\ 
      $\# r\ge 3$ &  0 &  0 &   0 &  0 \\
      \hline
      $\#$        &  2 &  1 & 235 & 35 \\ 
      \hline  
    \end{tabular}    
    
    \caption{Locality constants for $6$-dimensional binary codes with small lengths.}
    \label{table_q_2_k_6}
  \end{center}
\end{table}

\begin{table}
  \begin{center}
    \begin{tabular}{l|rrrrrrrrrrrrrrrrrrrrrr}
      \hline
      $n$         & 15 &    16 &     17 & 16 &  17 &    18 & 18 & 19 &     20 & 19 & 20 & 23 &            24 \\ 
      $d$         &  5 &     5 &      5 &  6 &   6 &     6 &  7 &  7 &      7 &  8 &  8 &  9 &             9 \\
      \hline
      %% $\# r=1$    &  0 &     0 &      0 &  0 &   0 &     0 &  0 &  0 &      0 &  0 &  0 &  0 & 0             \\
      $\# r=2$    &  0 &     0 &     48 &  0 &   0 &    25 &  0 &  0 &     23 &  0 &  1 &  0 & $\ge$ 411     \\ 
      $\# r\ge 3$ &  6 &  4013 & 459067 &  3 & 377 & 76922 &  2 & 82 & 207117 &  1 & 25 & 29 & $\ge$ 3572275 \\
      \hline
      $\#$        &  6 &  4013 & 459115 &  3 & 377 & 76947 &  2 & 82 & 207140 &  1 & 26 & 29 & $\ge$ 3572686 \\ 
      \hline  
    \end{tabular}
    
    \smallskip
    
    \begin{tabular}{l|rrrrrrrrrrrrrrrrrrrrrr}
      \hline
      $n$         & 24 &     25 & 26 &   27 & 27 &  28 & \\    
      $d$         & 10 &     10 & 11 &   11 & 12 &  12 & \\
      \hline
      %% $\# r=1$    &  0 &      0 &  0 &    0 &  0 &   0 & \\
      $\# r=2$    &  0 &    502 &  0 &   40 &  0 &   7 & \\
      $\# r\ge 3$ &  6 & 188064 &  2 & 1598 &  1 & 122 & \\
      \hline
      $\#$        &  6 & 188566 &  2 & 1638 &  1 & 129 & \\
      \hline  
    \end{tabular}
    
    \caption{Locality constants for $7$-dimensional binary codes with small lengths.}
    \label{table_q_2_k_7}
  \end{center}
\end{table}

%% \section{Conclusion}
%% \label{sec_conclusion}
    
%% \section*{Acknowledgments}
%% The author thanks Gianira Alfarano, Anurag Bishnoi, Jozefien D'haeseleer, Dion Gijswijt, Alessandro Neri, Sven Polak, and Martin Scotti for many helpful remarks on an earlier version of this paper, 
%% which originally started to investigate so-called trifferent codes, see \cite{trifferent_kurz}.

%%\bibliographystyle{alphaurl}
%%\bibliography{lrc}

\begin{thebibliography}{GWFHH19}

\bibitem[ABH{\etalchar{+}}18]{agarwal2018combinatorial}
Abhishek Agarwal, Alexander Barg, Sihuang Hu, Arya Mazumdar, and Itzhak Tamo.
\newblock Combinatorial alphabet-dependent bounds for locally recoverable
  codes.
\newblock {\em IEEE Transactions on Information Theory}, 64(5):3481--3492,
  2018.

\bibitem[APD{\etalchar{+}}13]{asteris2013xoring}
Megasthenis Asteris, Dimitris Papailiopoulos, Alexandros~G Dimakis, Ramkumar
  Vadali, Scott Chen, and Dhruba Borthakur.
\newblock {X}{O}{R}ing elephants: {N}ovel erasure codes for big data.
\newblock {\em Proceedings of the VLDB Endowment}, 6(5):325--336, 2013.

\bibitem[BBK21]{bouyukliev2021computer}
Iliya Bouyukliev, Stefka Bouyuklieva, and Sascha Kurz.
\newblock Computer classification of linear codes.
\newblock {\em IEEE Transactions on Information Theory}, 67(12):7807--7814,
  2021.

\bibitem[BCP97]{MR1484478}
Wieb Bosma, John Cannon, and Catherine Playoust.
\newblock The {M}agma algebra system. {I}. {T}he user language.
\newblock {\em J. Symbolic Comput.}, 24(3-4):235--265, 1997.
\newblock Computational algebra and number theory (London, 1993).
\newblock URL: \url{http://dx.doi.org/10.1006/jsco.1996.0125}, \href
  {https://doi.org/10.1006/jsco.1996.0125} {\path{doi:10.1006/jsco.1996.0125}}.

\bibitem[BE97]{bierbrauer1997family}
J{\"u}rgen Bierbrauer and Yves Edel.
\newblock A family of $2$-weight codes related to {B}{C}{H}-codes.
\newblock {\em Journal of Combinatorial Designs}, 5(5):391--396, 1997.

\bibitem[BJ01]{bouyukliev2001optimal}
Iliya Bouyukliev and David~B. Jaffe.
\newblock Optimal binary linear codes of dimension at most seven.
\newblock {\em Discrete Mathematics}, 226(1-3):51--70, 2001.

\bibitem[BJV00]{bouyukhev2000smallest}
Iliya Bouyukliev, David~B. Jaffe, and Vesselin Vavrek.
\newblock The smallest length of eight-dimensional binary linear codes with
  prescribed minimum distance.
\newblock {\em IEEE Transactions on Information Theory}, 46(4):1539--1544,
  2000.

\bibitem[BM73]{baumert1973note}
L.D. Baumert and R.J. McEliece.
\newblock A note on the {G}riesmer bound.
\newblock {\em IEEE Transactions on Information Theory}, pages 134--135, 1973.

\bibitem[CHL07]{chen2007maximally}
Minghua Chen, Cheng Huang, and Jin Li.
\newblock On the maximally recoverable property for multi-protection group
  codes.
\newblock In {\em 2007 IEEE International Symposium on Information Theory},
  pages 486--490. IEEE, 2007.

\bibitem[CM15]{cadambe2015bounds}
Viveck~R Cadambe and Arya Mazumdar.
\newblock Bounds on the size of locally recoverable codes.
\newblock {\em IEEE Transactions on Information Theory}, 61(11):5787--5794,
  2015.

\bibitem[GC14]{goparaju2014binary}
Sreechakra Goparaju and Robert Calderbank.
\newblock Binary cyclic codes that are locally repairable.
\newblock In {\em 2014 IEEE International Symposium on Information Theory},
  pages 676--680. IEEE, 2014.

\bibitem[GFHWH17]{grezet2017binary}
Matthias Grezet, Ragnar Freij-Hollanti, Thomas Westerb{\"a}ck, and Camilla
  Hollanti.
\newblock On binary matroid minors and applications to data storage over small
  fields.
\newblock In {\em International Castle Meeting on Coding Theory and
  Applications}, pages 139--153. Springer, 2017.

\bibitem[GFHWH19]{grezet2019alphabet}
Matthias Grezet, Ragnar Freij-Hollanti, Thomas Westerb{\"a}ck, and Camilla
  Hollanti.
\newblock Alphabet-dependent bounds for linear locally repairable codes based
  on residual codes.
\newblock {\em IEEE Transactions on Information Theory}, 65(10):6089--6100,
  2019.

\bibitem[GHSY12]{gopalan2012locality}
Parikshit Gopalan, Cheng Huang, Huseyin Simitci, and Sergey Yekhanin.
\newblock On the locality of codeword symbols.
\newblock {\em IEEE Transactions on Information Theory}, 58(11):6925--6934,
  2012.

\bibitem[GJR23]{gruica2023lrcs}
Anina Gruica, Benjamin Jany, and Alberto Ravagnani.
\newblock {L}{R}{C}s: Duality, {L}{P} bounds, and field size.
\newblock {\em arXiv preprint 2309.03676}, 2023.

\bibitem[Gri60]{griesmer1960bound}
James~H. Griesmer.
\newblock A bound for error-correcting codes.
\newblock {\em IBM Journal of Research and Development}, 4(5):532--542, 1960.

\bibitem[GWFHH19]{grezet2019uniform}
Matthias Grezet, Thomas Westerb{\"a}ck, Ragnar Freij-Hollanti, and Camilla
  Hollanti.
\newblock Uniform minors in maximally recoverable codes.
\newblock {\em IEEE Communications Letters}, 23(8):1297--1300, 2019.

\bibitem[Hel83]{helleseth1983new}
Tor Helleseth.
\newblock New constructions of codes meeting the {G}riesmer bound.
\newblock {\em IEEE Transactions on Information Theory}, 29(3):434--439, 1983.

\bibitem[HSX{\etalchar{+}}12]{huang2012erasure}
Cheng Huang, Huseyin Simitci, Yikang Xu, Aaron Ogus, Brad Calder, Parikshit
  Gopalan, Jin Li, and Sergey Yekhanin.
\newblock Erasure coding in windows azure storage.
\newblock In {\em Proceedings of the 2012 USENIX conference on Annual Technical
  Conference}, pages 15--26, 2012.

\bibitem[HXC16]{hao2016some}
Jie Hao, Shu-Tao Xia, and Bin Chen.
\newblock Some results on optimal locally repairable codes.
\newblock In {\em 2016 IEEE International Symposium on Information Theory
  (ISIT)}, pages 440--444. IEEE, 2016.

\bibitem[HXC17]{hao2017optimal}
Jie Hao, Shu-Tao Xia, and Bin Chen.
\newblock On optimal ternary locally repairable codes.
\newblock In {\em 2017 IEEE International Symposium on Information Theory
  (ISIT)}, pages 171--175. IEEE, 2017.

\bibitem[HXS{\etalchar{+}}20]{hao2020bounds}
Jie Hao, Shu-Tao Xia, Kenneth~W Shum, Bin Chen, Fang-Wei Fu, and Yixian Yang.
\newblock Bounds and constructions of locally repairable codes: parity-check
  matrix approach.
\newblock {\em IEEE Transactions on Information Theory}, 66(12):7465--7474,
  2020.

\bibitem[HYUS15]{huang2015cyclic}
Pengfei Huang, Eitan Yaakobi, Hironori Uchikawa, and Paul~H. Siegel.
\newblock Cyclic linear binary locally repairable codes.
\newblock In {\em 2015 IEEE Information Theory Workshop (ITW)}, pages 1--5.
  IEEE, 2015.

\bibitem[HYUS16]{huang2016binary}
Pengfei Huang, Eitan Yaakobi, Hironori Uchikawa, and Paul~H. Siegel.
\newblock Binary linear locally repairable codes.
\newblock {\em IEEE Transactions on Information Theory}, 62(11):6268--6283,
  2016.

\bibitem[KY21]{kurz2021pir}
Sascha Kurz and Eitan Yaakobi.
\newblock {P}{I}{R} codes with short block length.
\newblock {\em Designs, Codes and Cryptography}, 89:559--587, 2021.

\bibitem[LXY18]{luo2018optimal}
Yuan Luo, Chaoping Xing, and Chen Yuan.
\newblock Optimal locally repairable codes of distance $3$ and $4$ via cyclic
  codes.
\newblock {\em IEEE Transactions on Information Theory}, 65(2):1048--1053,
  2018.

\bibitem[Mar97]{maruta1997achievement}
Tatsuya Maruta.
\newblock On the achievement of the griesmer bound.
\newblock {\em Designs, Codes and Cryptography}, 12:83--87, 1997.

\bibitem[PKLK12]{prakash2012optimal}
N.~Prakash, Govinda~M. Kamath, V.~Lalitha, and P.~Vijay Kumar.
\newblock Optimal linear codes with a local-error-correction property.
\newblock In {\em 2012 IEEE International Symposium on Information Theory
  Proceedings}, pages 2776--2780. IEEE, 2012.

\bibitem[SS65]{solomon1965algebraically}
Gustave Solomon and Jack~J. Stiffler.
\newblock Algebraically punctured cyclic codes.
\newblock {\em Information and Control}, 8(2):170--179, 1965.

\bibitem[TB14]{tamo2014family}
Itzhak Tamo and Alexander Barg.
\newblock A family of optimal locally recoverable codes.
\newblock {\em IEEE Transactions on Information Theory}, 60(8):4661--4676,
  2014.

\bibitem[TBGC15]{tamo2015cyclic}
Itzhak Tamo, Alexander Barg, Sreechakra Goparaju, and Robert Calderbank.
\newblock Cyclic {L}{R}{C} codes and their subfield subcodes.
\newblock In {\em 2015 IEEE International Symposium on Information Theory
  (ISIT)}, pages 1262--1266. IEEE, 2015.

\bibitem[TPD16]{tamo2016optimal}
Itzhak Tamo, Dimitris~S. Papailiopoulos, and Alexandros~G. Dimakis.
\newblock Optimal locally repairable codes and connections to matroid theory.
\newblock {\em IEEE Transactions on Information Theory}, 62(12):6661--6671,
  2016.

\bibitem[vT81]{van1981smallest}
Henk~C.A. van Tilborg.
\newblock The smallest length of binary $7$--dimensional linear codes with
  prescribed minimum distance.
\newblock {\em Discrete Mathematics}, 33(2):197--207, 1981.

\bibitem[WFHEH16]{westerback2016combinatorics}
Thomas Westerb{\"a}ck, Ragnar Freij-Hollanti, Toni Ernvall, and Camilla
  Hollanti.
\newblock On the combinatorics of locally repairable codes via matroid theory.
\newblock {\em IEEE Transactions on Information Theory}, 62(10):5296--5315,
  2016.

\bibitem[XKG22]{xi2022optimal}
Yuanxiao Xi, Xiangliang Kong, and Gennian Ge.
\newblock Optimal quaternary locally repairable codes attaining the
  {S}ingleton-like bound.
\newblock {\em arXiv preprint 2206.05805}, 2022.

\end{thebibliography}

\newcommand{\etalchar}[1]{$^{#1}$}

\end{document}